\documentclass[letterpaper,11pt,oneside]{article}
\usepackage[T1]{fontenc}
\usepackage[latin1]{inputenc}
\usepackage{epsfig}
\usepackage{amsmath}
\usepackage{amssymb}
\usepackage{mathtools}
\usepackage{amsthm}
\usepackage[active]{srcltx}
\usepackage{dsfont}
\usepackage[titles]{tocloft}
\usepackage{color}
\numberwithin{equation}{section}
\newtheorem{theoreme}{Theorem}[section]
\newtheorem{proposition}{Proposition}[section]
\newtheorem{remarque}[theoreme]{Remark}

\newtheorem{claim}{Claim}
\newtheorem{example}{Example}
\newtheorem{lemme}{Lemma}

\newtheorem{corollaire}[theoreme]{Corollary}
\newtheorem{definition}{Definition}

\definecolor{darkblue}{rgb}{0,0,0.7} 

\newcommand{\RR}{\ensuremath{\mathbb R}}

\newcommand{\PP}{\ensuremath{\mathbb P}}
\newcommand{\EE}{\ensuremath{\mathbb E}}
\newcommand{\NN}{\ensuremath{\mathbb N}}
\newcommand{\R}{\ensuremath{\mathcal R}}

\newcommand{\I}{\ensuremath{\mathcal I}}
\newcommand{\T}{\ensuremath{\mathcal T}}
\newcommand{\J}{\ensuremath{\mathcal J}}
\newcommand{\F}{\ensuremath{\mathcal F}}

\newcommand{\ind}{\ensuremath{\mathds 1}}

\newcommand{\Id}{\ensuremath{\operatorname{Id}}}

\newcommand{\om}{\omega}
\newcommand{\Om}{\Omega}
\newcommand{\De}{\Delta}
\newcommand{\la}{\lambda}

\newcommand{\de}{\delta}
\newcommand{\be}{\beta}
\newcommand{\ep}{\varepsilon}
\newcommand{\al}{\alpha}

\newcommand{\ga}{\gamma}
\newcommand{\Ga}{\Gamma}

\newcommand{\limla}{\lim_{\lambda \to 0}}

\newcommand{\Ph}{\widehat{P}}
\newcommand{\Xh}{\widehat{X}}

\newcommand{\pih}{\widehat{\pi}}

\newcommand{\limn}{\lim_{n\to\infty}}

\def\e{\textrm{e}}
\usepackage{tikz}
\usetikzlibrary{shapes}
\title{Occupation measures arising in finite stochastic games}

\author{Bruno \textsc{Jaffuel}\\[0.25cm]
\small Universit\'e Pierre et Marie Curie (Paris 6)\\
\small jaffuel@phare.normalesup.org\\[0.5cm]
Miquel \textsc{Oliu-Barton}\\[0.25cm]
\small Universit\'e Pierre et Marie Curie (Paris 6)\\ 
\small miquel.oliu.barton@normalesup.org 
}

\date{December 27, 2013}

\begin{document}
\maketitle

\abstract{Shapley~\cite{shapley53} introduced two-player zero-sum discounted stochastic games, henceforth stochastic games, a model where a state variable follows a two-controlled Markov chain, the players receive rewards at each stage which add up to $0$, and each maximizes the normalized $\la$-discounted sum of stage rewards, for some fixed discount rate $\la\in(0,1]$. 
In this paper, we study asymptotic occupation measures arising in these games, as the discount rate goes to $0$. 

\section{Introduction}
Let $\Om$ be a finite set of states and let $Q$ be a stochastic matrix over $\Om$.
A classical result is the existence of the weak ergodic limit 
$\Pi:=\limn \frac{1}{n} \sum_{m=0}^{n-1} Q^m$.
The sensitivity of the ergodic limit to small perturbations of $Q$ goes back to~\cite{schweitzer68}. 
The simplest case is that of a linear perturbation of $Q$, $Q_\ep:=\frac{1}{1+\ep}(Q+\ep P)$ $(\ep\geq 0)$, where 
$P$ is another stochastic matrix over $\Om$. A perturbation is said to be \emph{regular} if the recurrence classes 
remain constant in a neighbourhood of $0$. 
When the perturbation
is not regular, the ergodic limit of $Q_\ep$ may fail to converge, as $\ep$ tends to $0$, to
that of $Q=Q_0$.
\begin{example}\label{ex1} Let $\Om=\{1,2\}$, $Q=\Id$, and $P=(\begin{smallmatrix}
    0 & 1\\ 1 & 0
   \end{smallmatrix})$, so that $Q_\ep=\frac{1}{1+\ep}(\begin{smallmatrix}
    1 & \ep\\ \ep & 1
   \end{smallmatrix})$ for $\ep\geq 0$. Then $\Pi_\ep=\limn \frac{1}{n} \sum_{m=0}^{n-1} {Q_\ep}^m=(\begin{smallmatrix}1/2& 1/2\\ 1/2& 1/2\end{smallmatrix})$ for all $\ep>0$, whereas 
   $\Pi_0=\Id$.
\end{example}
The study of regular and nonregular perturbations has been widely treated in the litterature. 
The aim of this paper is to study a Markov chain perturbation problem arising in the asymptotic study 
of two-person zero-sum stochastic games.
An important aspect in this model (see Section~\ref{sto}) is the discount rate $\la>0$ which models
the impatience of the players. As a consequence, we will consider from now on the Abel mean $\sum_{m\geq 0}\la(1-\la)^m Q^m$ instead of the Cesaro mean.
Notice that, for any fixed stochastic matrix $Q$, Hardy-Littlewood's Tauberin Theorem~\cite{HL26} gives the equality:
$$\Pi=\limn \frac{1}{n} \sum_{m=0}^{n-1} Q^m=\limla \sum_{m\geq 0}\la(1-\la)^m Q^m.$$
Suppose now that $(Q_\la)_\la$ is a family of stochastic matrices, for $\la\in[0,1]$. It is not hard to see that if 
$Q_\la$ is a regular perturbation of $Q_0$ in a neighborhood of $0$, then again: 
$$\Pi=\limla \sum_{m\geq 0}\la(1-\la)^m Q_\la^m.$$
It is enough to write:
\begin{equation}\sum_{m\geq 0} \la(1-\la)^{m}Q_\la^m= \sum_{m\geq 0 } \la^2 (1-\la)^{m}(m+1)
\left(\frac{1}{m+1}\sum_{k=0}^{m} Q_\la^k\right),\end{equation}
and use the fact that $\lim_{m\to\infty} \frac{1}{m+1}\sum_{k=0}^{m} Q_\la^k=\Pi_\la$, which converges to $\Pi$ as $\la$ tends to $0$.
The case of non-regular perturbation is most interesting, as shows the following example. 

\begin{example}\label{ex2} For any $a\geq 0$, 
let $Q_\la(a):=\begin{pmatrix}1-\la^a& \la^a \\ \la^a & 1-\la^a\end{pmatrix}$. Note that $Q_\la(0)=(\begin{smallmatrix}0 & 1\\1 & 0\end{smallmatrix})$ is
periodic and that, for any $a>0$, $Q_\la(a)$ is a non-regular perturbation
of $Q_0(a)=\Id$.
Computation yields: 
$$\limla \sum_{m\geq 0} \la (1-\la)^{m} Q_\la(a)^{m}=
\begin{cases}
\begin{pmatrix} 1/2 & 1/2\\ 1/2 & 1/2\end{pmatrix}, & \text{if } 0\leq a<1;\\
\begin{pmatrix} 2/3 & 1/3\\ 1/3 & 2/3\end{pmatrix}, & \text{if } a=1;\\ 
\Id, & \text{if } a>1.
 \end{cases}$$
The case $a=1$ appears as the critical value and can be explained  by the fact that the perturbation and the discount rate are ``of same order''. 
 \end{example}                                                  
The convergence, as $\la$ tends to $0$, of $\sum_{m\geq 0} \la(1-\la)^{m}Q_\la^m$ needs some regularity of the
family $(Q_\la)_\la$. The following condition is a natural regularity requirement:  \\ 
\textbf{Assumption 1:} There exist $c_{\om,\om'}, e_{\om,\om'}\geq 0$ $(\om,\om' \in \Om)$ such that:
\begin{equation}\label{ass}Q_\la(\om,\om')\sim_{\la\to 0}c_{\om,\om'}\la^{e_{\om,\om'}}.
\end{equation}
The constants $c_{\om,\om'}$ and $e_{\om,\om'}$ are referred as the coefficient
and the exponent of the transition $Q_\la(\om,\om')$. 
By convention, we set $e_{\om,\om'}=\infty$ whenever $c_ {\om,\om'}=0$. \\

Assumption 1 holds in the rest of the paper. 
Note that a perturbation satisfying this assumption can be regular or non-regular. 

\subsection{From stochastic games to occupation measures}\label{sto} 
Two-person zero-sum stochastic games were introduced by Shapley~\cite{shapley53}. They are
%
are described by a $5$-tuple $(\Om, \I,\J,q,g)$, where $\Om$ is a finite set of states, $\I$ and $\J$ are finite sets of actions,
$g:\Om \times \I \times \J \to [0,1]$ is the payoff, $q: \Om \times \I \times \J \to \De(\Om)$ the transition and,
for any finite set $X$, $\De(X)$ denotes the set of probability distributions over $X$. The functions
$g$ and $q$ are bilinearly extended to $\Om\times \De(\I)\times \De(\J)$.
 The stochastic game with initial state $\om\in\Om$ and discount rate $\la\in(0,1]$ is denoted by 
$\Ga_\la(\om)$ and is played as follows:  at stage $m\geq1$, knowing the current state $\om_m$, the players choose actions  
$(i_m,j_m)\in \I\times \J$; their choice produces a stage payoff $g(\om_m,i_m,j_m)$ and influences the transition: a new state $\om_{m+1}$ is chosen according to the probability distribution $q(\cdot|\om_m,i_m,j_m)$.
At the end of the game, player $1$ receives $\sum\nolimits_{m \geq 1}\la(1-\la)^{m-1} g(\omega_m,i_m,j_m)$ from player $2$.
The game $\Ga_\la(\om)$ has a value $v_\la(\om)$, and the vector 
$v_\la=(v_\la(\om))_{\om\in\Om}$ is the unique fixed point of the so-called Shapley operator
\cite{shapley53}: $\Phi_\la:\RR^\Om\to \RR^\Om$,
\begin{equation}\label{shapley} \Phi_\la(f)(\om)=\mathrm{val}_{(s,t)\in \De(\I)\times \De(\J)}\left
 \{\la g(\om,s,t)+(1-\la)\EE_{q(\cdot|\om,s,t)}[f(\widetilde{\om})]\right \}.
 \end{equation}
From \eqref{shapley}, one deduces the existence of optimal stationary strategies $x:\Om\to \De(\I)$ and $y:\Om\to \De(\J)$. 
 The convergence of the discounted values as $\la$ tends to $0$ 
 is due to Bewley and Kohlberg~\cite{BK76}. An alternative proof was recently obtained in~\cite{OB14}.
Let $v:=\limla v_\la\in \RR^\Om$ be the vector of limit values.

If both players play stationary strategies $x$ and $y$ 
 in $\Ga_\la$, then every visit to $\om$ produces an expected payoff of $g(\om,x(\om),y(\om))$, and a transition 
 $Q(\om,\cdot):=q(\cdot|\om,x(\om),y(\om))$. Thus, the expected payoff induced by $(x,y)$, denoted by $\ga_\la(\om,x,y)$, satisfies:
\begin{equation}\label{alt}
\ga_\la(\om,x, y)=\sum_{m\geq 1}\la(1-\la)^{m-1} Q^{m-1}(\om,\om')\sum_{\om' \in \Om} g(\om',x(\om'),y(\om')).
\end{equation}
Consider a family of stationary strategies $(x_\la,y_\la)_\la$, and let $(g_\la)_\la$ and $(Q_\la)_\la$ be the corresponding families of state-payoffs and transition matrices. 
Provided that the limits exist, the boundedness of $\sum_{m\geq 1}\la(1-\la)^{m-1} Q_\la^{m-1}$ yields:
\begin{equation}\label{limga}
\limla \ga_\la(\cdot, x_\la,y_\la)=\left(\limla \sum_{m \geq 1}\la (1-\la)^{m-1} Q_\la^{m-1}\right)\left(\limla g_\la\right),
\end{equation}
where the existence of the limits clearly 
requires 
some ``regularity'' of $(x_\la)_\la$ and $(y_\la)_\la$ in a neighbourhood of $0$. 
\begin{definition} A family $(x_\la)_\la$ of stationary strategies of player $1$ is: 
\begin{itemize}
 \item [$(i)$] \emph{Regular} if there exists coefficients $c_{\om,i}, c_{\om,j}>0$ and exponents $e_{\om,i},e_{\om,j}\geq 0$ such that ${x}_\la^i(\om)
                              \sim_{\la\to 0}c_{\om,i} \la^{e_{\om,i}}$,
                              for all  $\om\in \Om$, $i\in \I$ and $j\in \J$.
\item[$(ii)$] \emph{Asymptotically optimal}
if, for any $j:\Om\to \J$ pure stationray strategy of player $2$  (or equivalently, for any $y:\Om\to \De(\J)$, or any $(y_\la)_\la$): 
$$\liminf_{\la \to 0} \ga_\la(\om,x_\la,j)\geq v(\om).$$
\end{itemize}
\end{definition}
Similar definitions hold for families of stationary strategies of player $2$. 
Regular, asymptotically strategies exists~\cite{BK76}. 
Suppose that $(x_\la)_\la$ and $(y_\la)_\la$ are regular. A direct consequence is that $(Q_\la)_\la$ satisfies Assumption $1$.
On the other hand, the existence of $\limla g_\la$ is then straightforward. 
These observations motivate the study of $(Q_\la)_\la$ under Assumption $1$. We are interested in describing the distribution
over the state space at any fraction of the game $t\in [0,1]$, given a pair of regular stationary strategies.

\subsection{Main results}
Let $(Q_\la)_\la$ be a fixed family of stochastic matrices over $\Om$ satisfying Assumption $1$.
Let 
$(X^\la_m)_{m\geq 0}$ be a Markov chain 
with transition $Q_\la$. Extend the notation $X^\la_m$, 
which makes sense for integer times, to any real positive time
by setting $X^\la_t:=X^\la_{\lfloor t \rfloor}$, where $\lfloor t \rfloor=\max\{k\in \NN\, | \, t\geq k\}$. 
The process $(X_t)_{t>0}$ is a continuous time inhomogeneous Markov chain which jumps at integer times. 

For any $\om\in \Om$, and $\la\in(0,1]$, $\sum_{m\geq 1}\la(1-\la)^{m-1}Q_\la^{m-1}(\om,\om')$ is the expected time spent in state $\om'$, 
starting from $\om$,
if the weight given to stage $m$ is $\la(1-\la)^{m-1}$. Thus, for any $\la$ and $n\in \NN$, 
the weight given to the first $n$ stages for a discount rate $\la$ is $$\varphi(\la,n):=\sum_{k=1}^n \la(1-\la)^{m-1}\,.$$ 
In particular, note that $\limla \varphi(\la,\lfloor t/\la \rfloor)=1-\e^{-t}$ so that,
asymptotically, 
the first $\lfloor t/\la \rfloor$ stages represent a \emph{fraction} $1-\e^{-t}$ of the play 
(see Figure 1). We denote this fraction of the game by ``\emph{time $t$}'' and the limit, as $t$ tends to $0$, by ``\emph{time 0}''.
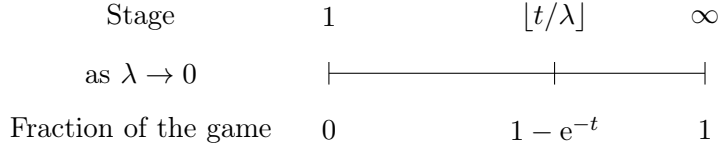
\begin{figure}\label{fig1}
\begin{center}
\begin{tikzpicture}
\draw (0,0) to (5,0);
\draw (0,-0.15) to (0,0.15);
\draw (5,-0.15) to (5,0.15);
\draw (3,-0.15) to (3,0.15);
\node   at (-2.5,0.75) {Stage};
\node   at (-2.5,0) {as $\la \to 0$};
\node   at (-2.5,-0.75) {Fraction of the game};
\node   at (0,-0.75) {$0$};
\node  at (5,-0.75) {$1$};
\node  at (3,-0.75) {$1-\e^{-t}$};
\node  at (0,0.75) {$1$};
\node  at (5,0.75) {$\infty$};
\node  at (3,0.75) {$\lfloor t/\la \rfloor $};
\end{tikzpicture}
\end{center}
\caption{Relation between the number of stages, the fraction of the game and the \emph{time}.} 
\end{figure}

In Sections~\ref{cases} and~\ref{general}, we study $P_t:=Q_\la^{\lfloor t/\la \rfloor}\in \De(\Om)$, for any $t>0$, interpreted 
as the (distribution of the) instantaneous position at time $t$. The existence of the limit is obtained, in some ``extended sense'' (see Section~\ref{instpos}) under Assumption 1.
Section~\ref{cases} explores two particular cases of the family $(Q_\la)_\la$: absorbing and critical, respectively. 
In these cases, an explicit computation of $P_t$ is obtained. Furthermore, we prove the convergence in distribution
of the Markov chains with transition $Q_\la$ to a Markov process in continuous time.

The general case is studied in Section~\ref{general}. 
For some $L\leq |\Om|$ and some set $\R=\{R_1,\dots,R_L\}$ of
subsets of $\Om$, we prove (see Theorem~\ref{main}) that the instantaneous position admits the following expression:
\begin{equation}\label{PAM}P_t=\mu \e^{A t}M,
\end{equation}
where $\mu:\Om\to \De(\R)$, $A:\R\times \R\to \RR$ and $M:\R\to \De(\Om)$. 
The elements of $\R$ are subsets of states such that, once they are reached, 
the probability of staying a strictly positive fraction of the play in them is strictly positive. 
They are the recurrent classes of a Markov chain defined in Section~\ref{proof}, which 
converges to a continuous time Markov process. Its infinitesimal generator is $A$, while $\mu$ represents the entrance laws to each of these subsets, and $M$ gives the frequency of visits to each state in the subsets of $\R$. 
From \eqref{PAM}, it follows (see Corollary~\ref{cumul}) that for any $t>0$,
$$\limla \sum_{m=1}^{\lfloor t/\la \rfloor} \la (1-\la)^{m-1} Q_\la^{m-1}= \mu\left(\int_0^t \e^{-s}\e^{As}\mathrm{d}s\right) M\,.$$
In Section~\ref{ill} we illustrate the computation of $\mu$, $A$ and $M$ in an example. Finally, using the fact that $A-\Id$ is invertible 
(by Gershgorin's Circle Theorem, for instance), we also obtain the following expression for the asymptotic payoff:
\begin{equation}\label{limga2}
\limla \ga_\la(\,\cdot\,, x_\la,y_\la)=\mu (A-\Id)^{-1}M g,
\end{equation}
where $g:=\limla g_\la\in \RR^{\Om}$.

\subsection{Notation} 
For any $\nu\in \De(\Om)$ and $B\subset \Om$ let $\nu(B):=\sum_{k\in B}\nu(k)$.
Let $P$ be some stochastic matrix over $\Om$, and let $(X_m)_{m\geq 0}$ be a Markov chain with transition $P$.
Let 
$B^c:=\Om\backslash B$ and, for any $k\in \Om$,
$k^c:=\{k\}^c$. In particular, $P(k,B)=\sum_{k'\in B} P(k,k')$.
Denote by $\Ph$ the stochastic matrix obtained by $P$ as follows. For any $k, k'\in \Om$, $k\neq k'$:
\begin{equation}
 \Ph(k,k')=\begin{cases} P(k,k')/P(k,k^c) & \text{ if } P(k,k^c)>0;\\
                         0 & \text{ if } P(k,k^c)=0,
           \end{cases}
\end{equation}
and $\Ph(k,k)=1-\sum_{k'\neq k}\Ph(k,k')$. Note that $\Ph(k,k)=0$ whenever $P(k,k^c)>0$, and that 
 $P$ and $\Ph$ have the same recurrence classes $R\in \R$. Denote by $\T$
the set of transient states. Let $\mu:\Om\to \R$ be, for any $k\in \Om$ and $R\in \R$, be the entrance probability 
from $k$ to the recurrence class $R$:
$$\mu(k,R):=\limn P^n(k,R)=\limn \Ph^n(k,R).$$
Let $\pi^R$ be the invariant measure of the restriction of $P$ to $R$, seen as a probability measure over $\Om$.
If the restriction $(X_m)_m$ to $R$ is $d$-periodic $(d\geq 2)$, let $\pi^{R}_k$ ($k=1,\dots,d$) be the invariant 
measure of $(X_{md+k})_m$. Note that, in this case, $\pi^R=\frac{1}{d}\sum_{k=1}^d \pi^{R}_k$. \\


 For any $\om \in \Om$, the probability of quitting $\om$ satisfies $Q_\la(\om,\om^c)\sim_{\la\to 0}c_\om \la^{e_\om}$, where:
\begin{equation}\label{ce2}e_{\om}:=\min\{e_{\om,\om'} \, |
\, \om'\neq \om, \ c_{\om,\om'}>0\} \quad \text{and} \quad c_\om:=
\sum_{\om'\neq \om} c_{\om,\om'} \ind_{\{e_{\om,\om'}=e_\om \}}.\end{equation}
For any $\om\in \Om$, let $\PP^\la_\om$ be the unique probability distribution over $\Om^\NN$ induced 
by $Q_\la$ and the initial state $\om$, i.e.
$\PP^\la_\om(X^\la_1=\om)=1$ and, 
for all $m\geq 1$ and $\om',\om''\in \Om$: $$\PP^\la_\om(X^\la_{m+1}=\om''|X^\la_m=\om')=Q_\la(\om',\om'').$$
For any $t>0$ and $\om\in \Om$, let $\PP^t_\om$ be a shortcut for $\PP^\la_{\om_0}(\cdot \, | \, X^\la_{t/ \la}=\om)$, 
for some fixed initial state $\om_0$. By the Markov property, the choice of the initial state is irrelevant.
Finally, let $\tau^\la_{B}:=\inf\{m\geq 1 \, | X^\la_m\in B\}$ be the first arrival to $B$ and let  
us end this section with a useful Lemma. 
\begin{lemme}\label{tec}
Let $P$ be an irreducible stochastic matrix over $\{1,\dots,n\}$ 
with invariant measure $\pi$, and let 
$S$ be a diagonal matrix with diagonal coefficients in $(0,1]$ such that
$P=\Id -S+S \Ph$. 
Then $\Ph$ is irreducible with invariant measure $\pih$ and:
$$\pi(k)=\frac{\pih(k)/S(k,k)}{\sum_{k'=1}^n
\pih(k')/S(k',k')}, \quad \text{for all } 1\leq k\leq n.$$
\end{lemme}
\begin{proof}
It is enough to check that the right-hand side of the equality is invariant
by $P$, which is equivalent to $\pih S^{-1} P=\pih S^{-1}$. We
easily compute:
$$\pih S^{-1} P=\pih S^{-1} (\Id -S+S \Ph)=\pih S^{-1}-\pih + \pih
\Ph=\pih S^{-1},$$
which completes the proof.
\end{proof}

\subsection{The instantaneous position at time $t$}\label{instpos}
For any $t>0$, let $Q_\la^{t/\la}:=Q_\la^{\lfloor t/\la \rfloor}$ when there is no risk of confusion. 
\begin{definition}
If the limit exists, let $P_t:\Om\to \De(\Om)$ such that for all $\om,\om'\in \Om$:
$$P_t(\om,\om')=\limla \PP^\la_\om(X^\la_{t/\la}=\om')=\limla Q_\la^{ t/\la}(\om,\om').$$ 
$P_t$ is the vector of (distributions of the) positions at time $t>0$. Let
$P_0:=\lim_{t\to 0} P_t$ be the position at time $0$.
\end{definition}
Assumption $1$ does not ensure the existence of $P_t$: 
consider a constant, periodic family $(Q_\la)_\la\equiv P:=(\begin{smallmatrix} 0 & 1\\ 1 & 0\end{smallmatrix})$. 
Then for any initial state:
\begin{equation}\label{cas}\PP^\la_{\om_0}(X^\la_{t/\la}=\om_0)=\begin{cases}
                                0 & \text{if } \lfloor t/\la \rfloor \equiv 0 \, (\mathrm{mod} \, 2)\\
                                1 & \text{if } \lfloor t/\la \rfloor \equiv 1 \, (\mathrm{mod} \, 2)
                               \end{cases}\end{equation}
so that the limit does not exist.
On the other hand, however, as $\la$ tends to $0$, the frequency of visits
to both states before stage $\lfloor t/\la\rfloor$ converges to $1/2$ for any $t>0$, and 
$$\limla \de_{\om_0}\frac{1}{2}\left( Q_\la^{\lfloor t/\la \rfloor}+ Q_\la^{\lfloor t/\la \rfloor+1}\right)=(1/2,1/2).$$ 
Moreover, both $\lfloor t/\la \rfloor$ and $\lfloor t/\la \rfloor+1$ represent the same fraction of the game.
These observations motivate the following definitions.
 Let $\widetilde{Q}_\la$ be such that $\widetilde{Q}_\la(\om,\om')\sim_{\la\to 0} Q_\la(\om,\om')\ind_{\{e_{\om,\om'}< 1\}}$, for all $\om, \om'\in \Om$, and let $N$ be the product of the periods of its recurrence classes. 
\begin{definition}\label{extpos} The \emph{extended position}  $\overline{P}_t$ at time $t\geq 0$ is obtained by averaging over $N$, i.e. 
$\overline{P}_t:\Om\to \De(\Om)$:
$$\overline{P}_t(\om,\om')=\limla \frac{1}{N}\sum_{m=0}^{N-1}  Q_\la^{\lfloor t/\la \rfloor+m}(\om,\om'), \quad \text{and } \overline{P}_0:=\lim_{t\to 0} \overline{P}_t.$$
We set $P_t:=\overline{P}_t$ when the latter exists. 
\end{definition}
Averaging over $N$, one avoids irrelevant pathologies related to periodicity. 
In the previous example, for instance, $N=2$ settled the problem. 
It clearly extends the previous definition since the existence of $P_t$ implies the existence of $\overline{P}_t$ and their equality. We prove in Theorem~\ref{main} that the latter always exists under Assumption 1. The following example shows why $N$ depends on $(\widetilde{Q}_\la)_\la$, rather than on $(Q_\la)_\la$.
\begin{example} \label{ex5} Fix $t>0$. Let
$Q_\la(a):=\begin{pmatrix}\la^a & 1-\la^a\\ 1-\la^a & \la^a\end{pmatrix}$, for $\la\in [0,1]$ and some $a\geq 0$.
Clearly, $Q_\la(a)$ is aperiodic for all $\la>0$. 
However, $P_t$ exists only for all $0\leq a<1$.
Indeed, consider the transition $Q_\la(a)$ for some $a>1$.
The probability that $X^\la_{m+1}=X^\la_m$ for some $m\leq \lfloor t/\la \rfloor $ is bounded by $(1-\la^a)^{\lfloor t/\la\rfloor }$ which converges to $0$ with $\la$. Thus, asymptotically, the chain behaves like the $2$-periodic matrix $(\begin{smallmatrix} 0 & 1\\ 1 & 0\end{smallmatrix})$ before time $t$, for any $t\geq 0$. 
\end{example} 
\section{Characterization of two special cases}\label{cases}
In this section we study two special families of stochastic matrices.
\begin{definition} A stochastic matrix over $\Om$ is \emph{absorbing} if $Q(\om,\om)=1$ for all
$\om \in \Om \backslash \{\om_0\}$. An absorbing matrix $Q$ will be identified with the vector $Q(\om_0,\cdot)\in \De(\Om)$. 
\end{definition}
\begin{definition} $(Q_\la)_\la$ is \emph{absorbing} if $Q_\la(\om,\om)=1$ for all $\la>0$, for all $\om\neq \om_0$.
\end{definition}
\begin{definition} $(Q_\la)_\la$ is \emph{critical} if $e_{\om,\om'}\geq 1$ for all $\om,\om'\in \Om$, $\om\neq \om$.
\end{definition}
\begin{definition} The \emph{infinitesimal generator} $A:\Om\times \Om \to \RR$ corresponding to a critical family $(Q_\la)_\la$ is defined as follows:
\begin{equation}\label{aaa}A(\om,\om'):=\limla \frac{Q_\la(\om,\om')}{\la} \ \ (\om'\neq \om) \quad \text{and} \quad A(\om,\om):=-\sum_{\om'\neq \om} A(\om,\om').
\end{equation}
\end{definition}
Note that $A=0$ if and only if $e_{\om,\om'}>1$ for all $\om\neq \om'$. Absorbing families are treated in Section~\ref{abs}, critical families in Section~\ref{lin}. 
In both cases, $P_t$ exists and its computation can be carried explicitly. 

\subsection{Absorbing case}\label{abs}
Let  $(Q_\la)_{\la\in (0,1]}$ be absorbing and let $\om_0$ be non-absorbing state. 
To simplify the notation, let $Q_\la$ stand for $Q_\la(\om_0,\cdot)$. For any $\om\neq \om_0$, 
let $c_{\om}:=c_{\om_0,\om}$ and $e_\om:=e_{\om_0,\om}$. Let also $e:= \min\{ e_{\om_0,\om} \, | \, \om\neq \om_0\}$ and 
$c:=\sum_{\om'\neq \om_0}c_{\om}\ind_{\{e_\om=e\}}$. Finally, let $P_t:=P_t(\om_0,\cdot)\in \De(\Om)$:

\begin{proposition}\label{om0} $P_t$ 
exists 
for any $t>0$. Moreover, one has:
  \begin{equation}\label{dic}P_t(\om_0)=
  \begin{cases} 
  1, & \text{if }e>1;\\ 
  0, & \text{if }0\leq e<1;\\ 
  \mathrm{e}^{-ct}, & \text{if }e=1.
  \end{cases}
  \end{equation}
 For any $\om\neq \om_0$: 
 \begin{equation}\label{dic2}
P_t(\om)=
 \begin{cases} 
 0, & \text{if }e>1;\\
 \frac{c_\om}{c}\ind_{\{e_\om=e\}}, & \text{if }0\leq e<1;\\
(1-\mathrm{e}^{-ct})\frac{c_\om}{c}\ind_{\{e_\om=e\}}, & \text{if }e=1.
 \end{cases} 
 \end{equation} 
\end{proposition}
\begin{proof} The equalities in \eqref{dic} are immediate since, by the definition of $e$:
$$\PP^\la_{\om_0}(X^\la_{t/\la}=\om_0)=(1-c\la^e+o(\la^e))^{\lfloor \frac{t}{\la} \rfloor}\sim_{\la \to 0}\e^{-ct\la^{e-1}}.$$ 
It follows that, for $e>1$, $\sum_{\om\neq \om_0}P_t(\om)=\limla \PP^\la_{\om_0}(X^\la_{t/\la}\neq \om_0)=0$, so that 
$P_t(\om)=0$ for all $\om\neq \om_0$, in this case. Similarly if $e\leq 1$ then for any $\om\neq \om_0$:
\begin{eqnarray}\PP^\la_{\om_0}(X^\la_{t/\la}=\om)&=&\PP^\la_{\om_0}(X^\la_{t/\la}\neq \om_0)\PP^\la_{\om_0}(X^\la_{t/\la}=\om\,| \,
X^\la_{t/\la}\neq \om_0),\\
&=& \PP^\la_{\om_0}(X^\la_{t/\la}\neq \om_0)\frac{Q_\la(\om)}{\sum_{\om\neq \om_0} Q_\la(\om)}.
\end{eqnarray} 
Taking the limit, as $\la$ tends to $0$ gives \eqref{dic2}.
\end{proof}

We can clearly distinguish three cases, depending on $e$, as in Example~\ref{ex2}:
\begin{itemize}
 \item[$(a)$] \textit{Stable} ($e>1$). $P_t(\om_0)=1$ for all $t>0$, so that $\om_0$ is ``never'' left.
  \item[$(b)$] \textit{Unstable} ($0\leq e< 1$). $P_t(\om_0)=0$ 
  for all $t>0$, so that $\om_0$ is left ``immediately''.
\item[$(c)$] \textit{Critical} ($e=1$). $P_t(\om_0)\in(0,1)$ for all $t>0$. 
From \eqref{dic}, one deduces that $\om_0$ is left at time $t$ with probability (density) $c\e^{-ct}dt$.
%

\end{itemize}

\subsection{Critical case}\label{lin}
Let  $(Q_\la)_{\la\in (0,1]}$ be critical. The next result explains why the matrix $A$, defined in \eqref{aaa} is 
denoted the infinitesimal generator. 
\begin{proposition}\label{lp}
For any $t\geq 0$ and $h>0$ and $\om'\neq \om$ we have, as $h \to 0$:
\begin{itemize}
 \item[$(i)$] $\limla \PP^t_\om(X^\la_{(t+h)/\la }=\om)=1+A(\om,\om)h + o(h)$,
 \item[$(ii)$] $\limla \PP^t_\om(X^\la_{(t+h)/\la}=\om')=A(\om,\om')h + o(h)$.
\end{itemize}
\end{proposition}
\begin{proof} \textbf{Notation.} Define two deterministic times $T^\la_0:=t/\la$ and $T^\la_h:=(t+h)/\la$. 
For any $k\in \NN$, let $F^\la_{0,h}(k)$ be the event that $(X^\la_m)_{m\geq 1}$ changes $k$ times of state
in the interval $[T^\la_0,T^\la_h]$,  
and let $F^\la_{0,h}(k^+)=\bigcup_{\ell\geq k}F^\la_{0,h}(\ell)$.\\
Notice that, conditional to $\{X^\la_{t/\la}=\om\}$, the following disjoint union holds:
$$\{X^\la_{(t+h)/\la}=\om\}= F^\la_{0,h}(0) \cup \left (\{X^\la_{(t+h)/\la}=\om\}\cap F^\la_{0,h}(2^+)\right).$$
The following computation is straightforward: 
\begin{eqnarray}
 \limla \PP^t_\om(F^\la_{0,h}(0))&=& \limla \prod_{m=\lfloor t/\la \rfloor}^{\lfloor (t+h)/\la \rfloor}
 \PP^t_\om(X^\la_{m+1}=X^\la_m),\\
 &=& \limla \left(1-\sum_{\om'\neq \om} Q_\la(\om,\om')\right)^{h/\la},\\
  &=& \exp\left(-\sum_{\om\neq \om'}A(\om,\om')h\right),\\ \label{az}
 &=& 1 + A(\om,\om)h+o(h), \quad \text{ as } h\to 0.
\end{eqnarray}
On the other hand:
\begin{equation}\label{2jumps}\PP^t_\om(F^\la_{0,h}(2+))\leq \max_{\om'\in \Om} \PP^t_{\om'}(F^\la_{0,h}(1+))^2=\max_{\om'\in \Om}\left(1-\PP^t_{\om'}(F^\la_{0,h}(0)\right)^2.
\end{equation}
Therefore,  $\limla \PP^t_\om(F^\la_{0,h}(2+))=o(h)$ as $h$ tends to $0$ which, together with \eqref{az}, proves $(i)$. Similarly, conditional on $\{X^\la_{t/\la}=\om\}$:
$$\{X^\la_{(t+h)/\la }=\om')\}= \{X^\la_{(t+h)/\la}=\om'\} \cap \left(F^\la_{0,h}(1) \cup  F^\la_{0,h}(2+)\right),$$
so that by \eqref{2jumps}:
$$\limla \PP^t_\om(X^\la_{(t+h)/\la }=\om')=\limla \PP^t_\om(F^\la_{0,h}(1),\, X^\la_{(t+h)/\la }=\om')+o(h).$$
On the other hand, for any $\la$ and $m$: $$\PP^\la(X^\la_{m+1}=\om'|X^\la_m=\om, X^\la_{m+1}\neq \om)=\frac{Q_\la(\om,\om')}{Q_\la(\om,\om^c)}.$$
Finally, note that by \eqref{2jumps}, $\limla \PP^t_\om(F^\la_{0,h}(1))=\limla 1-\PP^t_\om(F^\la_{0,h}(0))+o(h)$, and that 
$\limla \frac{Q_\la(\om,\om')}{Q_\la(\om,\om^c)}=-\frac{A(\om,\om')}{A(\om,\om)}$ for all $\om\neq \om'$. 
Consequently, as $h$ tends to $0$:
\begin{eqnarray*}\label{ert}\limla \PP^t_\om(F^\la_{0,h}(1),\, X^\la_{(t+h)/\la}=\om')&=&\limla 
\frac{Q_\la(\om,\om')}{Q_\la(\om,\om^c)}\left(1-\PP^t_\om(F^\la_{0,h}(0))+o(h)\right),\\
&=&-\frac{A(\om,\om')}{A(\om,\om)}\left(1-\e^{A(\om,\om)h}\right),\\
&=&A(\om,\om')h+o(h).
\end{eqnarray*}
\end{proof}

\begin{corollaire}\label{limit} The processes $(X^\la_{t/\la})_{t\geq 0}$ converge, as $\la$ tends to $0$, to a Markov process $(Y_t)_{t\geq 0}$ with generator $A$.
\end{corollaire}
\begin{proof}
 The limit is identified by Proposition~\ref{lp}. The
 tightness is a consequence of the bound in Proposition~\ref{lp}-$(ii)$, which implies that for any $T>0$,
 uniformly in $\la>0$:
 $$\lim_{\ep\to 0} \PP\left(\exists t_1,t_2\in [0,T] \, | \, t_1<t_2<t_1+\ep,\, X^\la_{t_i^-/\la}\neq X^\la_{t_i^-/\la},\, i\in \{1,2\}
 \right)=0,$$
which is precisely the tightness criterion for c\`adl\`ag process with discrete values.
 \end{proof}

The following result is both a direct consequence of Proposition~\ref{lp} or Corollary~\ref{limit}.
\begin{corollaire}\label{crit} $P_t$ exists for any $t>0$ and satifies $P_t=\e^{At}$.
\end{corollaire}

 
\section{The general case}\label{general}
In this section we drop the assumption of $(Q_\la)_\la$ being critical or absorbing.
Let us start by noticing that in Proposition~\ref{lp}, the time $t/\la$ may be replaced by 
$t/\la\pm1/\la^\de$, for any $0<\de<1$. That is:
\begin{equation}\label{remgen}P_t(\om,\om')=\limla \PP^\la_\om\left(X^\la_{t/\la\pm1/\la^\de}=\om'\right)
=\e^{At}(\om,\om'). 
\end{equation}
Note that, as $\la$ tends to $0$, $t/\la\pm1/\la^\de$ also corresponds to time $t$. this remark gives 
an idea of 
the flexibility to the terminology ``position at time $t$''. 
Our main result is the following.

\begin{theoreme}\label{main} There exists $L\leq |\Om|$, subsets $\R=\{R_1,\dots,R_L\}$ of $\Om$, $\mu:\Om\to \De(\R)$,
$A:\R \times \R \to \RR$ and $M:\R\times \Om\to \De(\Om)$ such that ${P}_t=\mu \e^{At}M$, for all $t\geq 0$.
\end{theoreme}
The proof of this result is constructive, and is left to Section~\ref{proof}, together with an algorithm
for the computation of $L$, $\R$, $\mu$, $A$ and $M$. 
An illustration of the algorithm is provided in Section~\ref{ill} by means of an example. 
Note that if $Q_\la$ were critical, then the
results in Section~\ref{lin} yield Theorem~\ref{main} with $L=|\Om|$, $\R=\Om$, $\mu=\Id=M$ and $A$ is defined in 
\eqref{aaa}. 

The following two results are direct consequences of Theorem~\ref{main}. 
\begin{corollaire}\label{cumul} For any $t>0$:
$$\limla \sum_{m=1}^{\lfloor t/\la \rfloor} \la(1-\la)^{m-1}Q_\la^{m-1}=\mu \left(\int_0^t \e^{-s}\e^{As}ds\right) M .$$
In particular,
$\limla \sum_{m\geq 1}\la(1-\la)^{m-1}Q_\la^{m-1}=\mu(\Id-A)^{-1}M.$
\end{corollaire}
For any $t\in[0,1)$, let $p_t:=P_{-\ln(1-t)}$ be the position at the fraction $t$ of the game.
\begin{corollaire}\label{expr} Let $v_i$ be the eigenvalues of $A$ and let $m_i$ be the size of the Jordan box corresponding 
to $v_i$, in the canonical form of $A$. Then for any $t\in[0,1)$ and $\om\in \Om$, $p_t(\om)$ is linear in
$(1-t)^{-v_i}\ln(1-t)^{k}$, $0\leq k\leq m_i-1$.
\end{corollaire}
The analogue of Corollary~\ref{limit} holds here, yet with some slight modifications. Unlike in Section~\ref{lin}, it is not
the processes $X^\la$ which converge, but rather their restriction to the set $\R$ obtained in Theorem~\ref{main}.
The proof is then, word for word, as in Corollary~\ref{limit}.
\begin{definition} The restriction of $X^\la$ to $\R$ is: 
$$\Xh^\la_m:=\Phi(X^\la_{V^\la_m}), \quad  m\geq 1,$$ 
where $V^\la_m$ is the time of the $m$-th visit of $X^\la$ to $\R$ and $\Phi$ is a mapping which
associates, to any state $\om\in \bigcup_{\ell=1}^L R_\ell\subset \Om$, the subset $R_\ell$ which contains it. 
Let 
$\Xh^\la_t:=\Xh^\la_{\lfloor t \rfloor}$, $t\geq 0$. 
\end{definition}
\begin{corollaire}\label{limit} Let $\R$, $\mu$ and $A$ be given in Theorem~\ref{main}. 
The processes $\Xh^\la$ converge, as $\la$ tends to $0$, to a Markov process with initial distribution $\mu$ and generator $A$. 
\end{corollaire}

\subsection{The order of a transition}
A natural way to rank the transitions of the Markov chains 
$(X_m^\la)_m$ is in terms of their (asymptotic) order of magnitude. For that prurpose, it is useful 
to define the following notion. 
\begin{definition} The \emph{order of the transition} from $\om$ to $\om'$
is defined as follows: 
$$r_{\om,\om'}:=\inf\left \{\al \geq 0 \, \bigg| \, \liminf_{\la\to 0} \PP^\la_\om\left(\tau^\la_{\om'}\leq \frac{1}{\la^\al}\right)>0\right\}.$$
\end{definition}
Let us present an example to illustrate this definition. 
\begin{example}\label{ex6} Let $0\leq a<b$, $\Om=\{1,\dots,n\}$ and suppose that
$Q_\la(1,2)=\la^a$, $Q_\la(1,3)=\la^b$ and $Q_\la(1,k)=0$ for all $k=3,\dots,n$. 
On the one hand, for any $\de<a$, 
$\PP^\la_1(\tau^\la_2>1/\la^\de)\geq (1-\la^a-\la^b)^{1/\la^\de}$. Taking the limit yields:
$$\limla \PP^\la_{1}(\tau^\la_2\leq 1/\la^\de)\leq 1- \limla (1-\la^a-\la^b)^{1/\la^\de}=0,$$
which implies that $r_{1,2}\geq a$. 
On the other hand, starting from state $1$, for any $m$:
$$\{\tau^\la_{1^c}\leq m\} \cap \{X^\la_{\tau^\la_{1^c}}=2\}\subset \{\tau^\la_2\leq m\}.$$
Taking the limit yields $r_{1,2}\leq a$ since:
$$\limla \PP^\la_{1}(\tau^\la_1\leq 1/\la^a)\geq \limla  \left(1-(1-\la^a-\la^{b})^{1/\la^a}\right)\frac{\la^a}{\la^a+\la^{b}}=1-1/e>0.$$
\end{example}
The previous example exhibits an explicit computation for the order of a transition. Note, however, that $r_{1,3}$ cannot be computed with the data we provided, for it depends on other entries of $Q_\la$. This is due to the fact that, conditional to leaving state $1$, the probability of going to state $2$ converges to $1$, so that the future behaviour of the chain depends on the vector $Q_\la(2,\cdot)$.
Let us give an example where the computation of the order of a transition is a bit more involved.
\begin{example}\label{ex} Let $0\leq a<b$, $c\geq 0$, and $\Om=\{1,2,3\}$. For any $\la\in[0,1]$, let:
$$Q_\la=\begin{pmatrix} 1-(\la^a+\la^b) & \la^a & \la^{b}\\
\la^{c}  & 1-\la^c & 0 \\  
 0& 0 & 1
        \end{pmatrix}.$$
The computation of $r_{1,2}=a$ and of $r_{2,1}=c$ is the same as in the previous example. Also, 
$r_{3,2}=r_{3,1}=\infty$ because $3$ is absorbing. Let us compute $r_{1,3}$ and $r_{2,3}$ heuristically.
In average, state $1$ is left after $1/(\la^a+\la^b)$ stages, then state $2$ is left after $1/\la^{c}$ stages, and we are back in state $1$ again. 
Hence, in average, an exit from state $1$ occurs every $\frac{1}{\la^a+\la^b}+\frac{1}{\la^c}$ stages.
Consequently, state $1$ is left $1/\la^b$ times, after:
$$\frac{\frac{1}{\la^a+\la^b}+\frac{1}{\la^c}}{\la^b}$$ 
stages, and thus the probability of reaching state $3$ is strictly positive. 
The relation $\frac{1}{\la^a+\la^b}+\frac{1}{\la^c}\sim_{\la\to 0} \frac{1}{\la^{\max\{a,c\}}}$, which holds because $a<b$, yields 
$r_{1,3}=r_{2,3}=\max\{a,c\}-b$. 
\end{example}
\subsection{Fastest and secondary transitions}\label{fast}
\begin{definition}
Let  $\al_1:=\min \{ e_{\om,\om'} \, | \,\om\neq \om'\in \Om\}$ be the \emph{order of the fastest transitions}.
A transition from $\om$ to $\om'$ is \emph{primary} if $e_{\om,\om'}=\al_1$. 
\end{definition} 
Note that if $\al_1\geq 1$, $Q_\la$ is critical. The results in Section~\ref{lin} apply and yield (see Corollary~\ref{crit}) 
Theorem~\ref{main} with $L=|\Om|$, $\R=\Om$, $\mu=M=\Id$ and $A$ defined in 
\eqref{aaa}. Suppose, on the contrary, that $\al_1<1$. Define a stochastic matrix $P^{[1]}_\la$ which is the restriction of $Q_\la$ to its fastest transitions. 
For any $\om\neq \om' \in \Om$ set: 
\begin{equation}
P^{[1]}_\la(\om,\om'):=\begin{cases} c_{\om,\om'}\la^{e_{\om,\om'}}, & \text{if } e_{\om,\om'}=\al_1;\\
0 & \text{otherwise};
\end{cases}
\end{equation}
and let $P^{[1]}_\la(\om,\om):=1-\sum_{\om'\neq \om}P^{[1]}_\la(\om,\om')$. 
Let $\R^{[1]}$ and $\T^{[1]}$ be, respectively, the set of its recurrence classes and transient states. Note that these sets are independent of $\la>0$. Let $\pi_\la^{[1],R}$ be the invariant measures of the restriction of $P^{[1]}_\la$ to the recurrence class $R\in \R^{[1]}$.
We can now define secondary transitions.
\begin{definition} The order of secondary transitions is:  
\begin{equation}\label{al2}\al_2:=\min \{e_{\om,\om'} \, | \, \om\in R, \, \om' \notin R, \ R\in \R^{[1]}\}.
\end{equation}
A transition from $\om$ to $\om'$ is secondary if $e_{\om,\om'}\geq \al_2$, and $\om\in R$, $\om'\notin R$, for some $R\in \
R^{[1]}$. 
\end{definition}
The definition of $P_\la^{[1]}$ implies that $\widehat{P^{[1]}_\la}$ is independent of $\la$ and has the same recurrence classes as $P_\la^{[1]}$. Denote this matrix by $\Ph^{[1]}$ and let $\pih^{[1],R}$ be the invariant measures of the restriction of $\Ph^{[1]}$ to $R$. 
The restriction of $P^{[1]}_\la$ to $R$ is irreducible and, consequently, we may apply Lemma~\ref{tec} with
the diagonal matrix $S^{[1],R}_\la$, defined for each $\om\in R$ as follows:
$$S^{[1],R}_\la(\om,\om):=P^{[1]}_\la(\om,\om^c).$$
Note that either $R=\{\om\}$ is a singleton and $S^{[1],R}_\la(\om,\om)=1$ or 
there are at least two states in $R$ and $S^{[1],R}_\la(\om,\om):=c_{\om}\la^{\al_1}$ for each $\om\in R$. The following result is thus a direct consequence of Lemma~\ref{tec}.
\begin{corollaire}\label{jaff}
Let $R\in \R$. 
Then there exist ${c}^{[1],R}(\om)>0$ 
($\om\in R$)
such that: 
$$\displaystyle \pi^{[1],R}_\la(\om)= \frac{\pih^{[1],R}(\om)/S^{[1],R}_\la(\om,\om)}
{\displaystyle \sum_{\om'\in R}\pih^{[1],R}(\om')/S^{[1],R}_\la(\om',\om')}= {c}^{[1],R}(\om).$$ 
\end{corollaire}
Since $\pi^{[1],R}_\la$ is independent of $\la$, we will denote it from now on simply by $\pi^{[1],R}$.
Conditional on having no transitions of order higher than $\al_1$ and on being in $R$, the frequency of visits to $\om\in R$ converges (exponentially fast) to $\pi^{[1],R}(\om)$. 
Consequently, the probability of a transition of higher order
going out from $R$ converges to $\sum_{\om\in R}\pi^{[1],R}(\om) Q_\la(\om,\cdot)$.
Aggregation is thus natural, in order to study phenomena of order strictly bigger than $\al_1$.

\subsection{Aggregating the recurrence classes}\label{aggr}
Aggregating the reccurrence classes stand to considering the state space $\Om^{[1]}:=\T^{[1]}
\cup \R^{[1]}$, i.e. an element $\om^{[1]}\in \Om^{[1]}$ is either a transient state $\om^{[1]}\in \T^{[1]}$ (in this case $\om^{[1]}=\om \in \Om$) or a recurrence class $\om^{[1]}= R\in \R^{[1]}$ (in this case $\om^{[1]} \subset \Om$).
In particular, the states of $\Om^{[1]}$ can be seen as a partition of the states of 
$\Om^{[0]}:=\Om$ (see Figure $3$ for an illustration). 
To avoid cumbersome notation, let $\om,\om'$ stand for states in $\Om^{[1]}$ when there is no confusion. One can then define an ``aggregated'' stochastic matrix $Q_\la^{[1]}$ over $\Om^{[1]}$ as follows.
\begin{equation}\label{q1}
Q^{[1]}_\la(\om,\om') :=\begin{cases} Q_\la(\om,\om') & \text{ if } \om,\om' \in \T^{[1]};\\
                 \sum_{z'\in \om'} Q_\la(\om,z') & \text{ if } \om \in \T^{[1]}, \text{ and } \om'\in \R^{[1]};\\
                 \sum_{z \in \om} \pi^{[1],\om}(z)Q_\la(z,\om') & \text{ if } \om \in \R^{[1]}, \text{ and } \om'\in \T^{[1]};\\
                 \sum_{z \in \om,\, z'\in \om'} \pi^{[1],\om}(z) Q_\la(z,z') & \text{ if } \om, \om' \in \R^{[1]}.
                \end{cases}
\end{equation}
Clearly, Assumption $1$ ensures the existence of $c^{[1]}_{\om,\om'}$ and $e^{[1]}_{\om,\om'}$ such that
$$Q^{[1]}_\la(\om,\om')\sim_{\la\to 0} 
c^{[1]}_{\om,\om'}\la^{e^{[1]}_{\om,\om'}}, \quad \forall \om,\om'\in \Om^{[1]}.$$
An explicit computation of $c^{[1]}_{\om,\om'}$ and $e^{[1]}_{\om,\om'}$ can be easily deduced from \eqref{q1}, 
in terms of the coefficients and exponents of $(Q_\la)_\la$ and of the invariant measures of $P^{[1]}_\la$.
The matrix $Q^{[1]}_\la$ arises by aggregating the state in the recurrence classes of $P^{[1]}_\la$. 
Define the entrance laws $\mu^{[1]}:\Om\to \De(\R^{[1]})$ as follows:
\begin{equation}\label{entr}\mu^{[1]}(\om,R):=\lim_{n\to \infty}(P^{[1]}_\la)^n(\om,R).
\end{equation}

If $\al_2\geq 1$, define $A^{[1]}:\R^{[1]}\times \R^{[1]}\to [0,\infty)$ 
the infinitesimal generator corresponding to $Q^{[1]}_\la$, as follows. 
For any $R,R'\in \R^{[1]}$: 
\begin{eqnarray}\label{defA1}A^{[1]}(R,R')
&:=& \limla \frac{1}{\la}\left(\sum_{\om \in \Om^{[1]}\backslash R} Q^{[1]}_\la(R,\om)\mu^{[1]}(\om,R')\right), 
\end{eqnarray}
and $A^{[1]}(R,R)=-\sum_{R'\neq R} A^{[1]}(R,R')$. Note that $A^{[1]}$ admits the following useful equivalent expression:
\begin{eqnarray}\label{eqA}A^{[1]}(R,R')
&=& \limla \frac{1}{\la}\left(\sum_{\om\in R,\, \om'\notin R}\pi^{[1],R}(\om)Q_\la(\om,\om')\mu^{[1]}(\om',R')\right).
\end{eqnarray}
Finally, let $M^{[1]}:\R^{[1]}\to \De(\Om)$ 
be such that, for any $\om\in R\in \R$:
\begin{equation}\label{m1} M^{[1]}(R,\om):=\pi^{[1],R}(\om). 
\end{equation}
\subsection{One-step dynamics}
The following intermediary step will be very useful in proving Theorem~\ref{main}. Assume in this section that
$\al_2\geq 1$. 
\begin{remarque} Notice that $\al_2>1$ if and only if $A^{[1]}=0$.
\end{remarque}

\begin{proposition}\label{lp2} If $\al_2\geq 1$, 
then $P_t= \mu^{[1]} \e^{A^{[1]}t}M^{[1]}, \ \forall t\geq 0.$
\end{proposition} 
Before getting into the proof, let us notice the following flexibility of the notion ``the position at time $t$''.
\begin{remarque}\label{flex} As in \eqref{remgen}, we will actually prove a slightly stronger statement: 
for any $R\in\R^{[1]}$, $t\geq 0$, $\om \in \Om$ and 
$\de'$ satisfying
$0\leq \al_1<\de'<1\leq \al_2$:
$$P_t(\om,R)=\limla \PP^\la_\om\left(X^\la_{T^\la_t\pm 1/\la^{\de'}}\in R \right)= \mu^{[1]} \e^{A^{[1]}t}(\om,R).$$
\end{remarque}

\noindent \emph{{Proof of Proposition~\ref{lp2}}}
Let $\de\in (\al_1,1)$, $t,h>0$ and $R,R'\in \R^{[1]}$ be fixed. 
The idea of the proof is similar to that of Proposition~\ref{lp}. One needs, however, to consider two more deterministic times, and take into account periodicity issues. 
Introduce some notation. \\
\textbf{Notation:} For any $h>0$, define four deterministic times (see Figure $2$):
$$T^\la_0:=\frac{t}{\la}, \quad T^\la_{\de}:= \frac{t}{\la}+\frac{1}{\la^\de}, \quad T^\la_{h-\de}:= \frac{t+h}{\la}-
\frac{1}{\la^\de}, \quad T^\la_h:=\frac{t+h}{\la}.$$
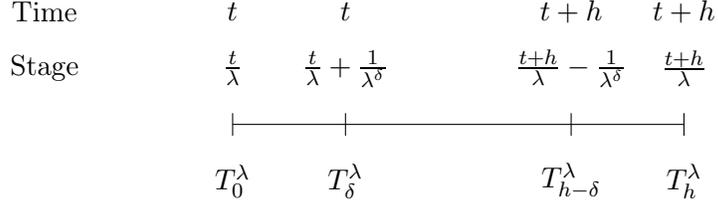
\begin{figure}\label{fig2}
\begin{center}
\begin{tikzpicture}
\draw (0,0) to (6,0);
\draw (0,-0.15) to (0,0.15);
\draw (1.5,-0.15) to (1.5,0.15);
\draw (4.5,-0.15) to (4.5,0.15);
\draw (6,-0.15) to (6,0.15);
\node   at (-2.5,0.75) {Stage};
\node   at (-2.5,1.5) {Time};
\node [scale=1] at (0,-0.75) {$T^\la_0$};
\node [scale=1]at (1.5,-0.75) {$T^\la_\de$};
\node [scale=1] at (4.5,-0.75) {$T^\la_{h-\de}$};
\node [scale=1] at (6,-0.75) {$T^\la_h$};

\node [scale=1]  at (0,0.75) {$ \frac{t}{\la}$};
\node [scale=1] at (1.5,0.75) {$ \frac{t}{\la}+\frac{1}{\la^\de}$};
\node [scale=1] at (4.5,0.75) {$ \frac{t+h}{\la}-\frac{1}{\la^\de}$};
\node [scale=1] at (6,0.75) {$ \frac{t+h}{\la} $};

\node [scale=1]  at (0,1.5) {$t$};
\node [scale=1]  at (1.5,1.5) {$t$};
\node [scale=1]  at (4.5,1.5) {$t+h$};
\node [scale=1]  at (6,1.5) {$t+h$};

\end{tikzpicture}

\end{center}
\caption{The four deterministic times, for fixed $t,h\geq 0$.}
\end{figure}
For any $k\in \NN$, and $\al,\be \in \{0,\de,h-\de,h\}$, denote by $F^\la_{\al,\be}(k)$
the event that $k$ secondary transitions 
of the Markov chain $X^\la$ 
occur in the interval $[T^\la_\al, T^\la_\be]$. 
Let $F^\la_{\al, \be}(k^+):=\bigcup_{\ell\geq k}F^\la_{\al,\be}(\ell)$. 
be the event corresponding to at least $k$ secondary transitions.
For any $k_1,k_2,k_3 \in \NN$, let $F^\la_{0,h}(k_1,k_2,k_3):= F^\la_{0,\de}(k_1)\cap F^\la_{\de,h-\de}(k_2)
\cap F^\la_{h-\de,h}(k_3)$. 
On the one hand, conditional to $X_{t/\la}^\la\in R$, since there is at least one secondary in order to leave 
the class $R$: 
$$\{X^\la_{(t+h)/\la }\in R'\}=\{X^\la_{(t+h)/\la }\in R'\}\cap F^\la_{0,h}(1^+).$$  
Moreover, the following disjoint union holds:
\begin{equation}\label{disj}F^\la_{0,h}(1^+)=F^\la_{0,h}(2^+) \cup F^\la_{0,h}(1,0,0)\cup F^\la_{0,h}(0,1,0)\cup F^\la_{0,h}(0,0,1). 
\end{equation}

\begin{claim}\label{easy} For any $\om\in \Om$ and $R\in \R^{[1]}$ one has, as $\la$ and $h$ tend to $0$:
\begin{itemize}
 \item[$(i)$] $\limla \PP^t_\om(F^\la_{0,h}(1^+))=O(h)$ and $\limla \PP^t_\om(F^\la_{0,h}(2^+))=o(h).$
 \item[$(ii)$] $\PP^t_\om(F^\la_{0,\de}(1^+))$, $\PP_\om^t(F^\la_{0,h}(1,0,0))$ and $\PP_\om^t(F^\la_{0,h}(0,0,1))$ are $O(\la^{\al_2-\de})$. 
\item [$(iii)$] $\limla \PP^t_{\om}(X^\la_{T^\la_{\de}}\in R)=\mu^{[1]}(\om,R)$. 
 \end{itemize}
\end{claim}
\noindent \textit{Proof of Claim~\ref{easy}.} $(i)$ Let $C\geq 0$ be such that the probability of a secondary transition from any state is smaller than $C\la^{\al_2}$. Then, the probability of having no secondary transition in $[T^\la_{0},T^\la_h]$ 
satisfies:
\begin{equation}\label{edr}\PP_\om^t\left(F^\la_{0,h}(0)\right)\geq (1-C\la^{\al_2})^{h/\la}\sim_{\la\to 0}\exp(-Ch\la^{\al_2-1}). 
\end{equation}
Taking the limit yields, due to $\al_2\geq 1$, that 
$\limla \PP_\om^t(F^\la_{0,h}(0))\geq 1-O(h)$, as $h$ tends to $0$.
But then $\limla \PP^t_\om(F^\la_{0,h}(1^+))=O(h)$ so that:
$$\limla \PP^t_\om(F^\la_{0,h}(2^+))\leq \limla \left( \max_{\om'\in \Om}P^t_{\om'}(F^\la_{0,h}(1^+))\right)^2 =o(h), \quad \text{as  }h\to 0.$$
$(ii)$ Clearly, $\PP_\om^t(F^\la_{0,h}(1,0,0))\leq \PP_\om^t(F^\la_{0,\de}(1))\leq 1-\PP_\om^t(F^\la_{0,\de}(0))=\PP^t_\om(F^\la_{0,\de}(1^+))$. As 
in \eqref{edr}, one has that: 
$$\PP_\om^t\left(F^\la_{0,\de}(0)\right)\geq (1-C\la^{\al_2})^{1/\la^\de}\sim_{\la\to 0}\exp(-C\la^{\al_2-\de})=1-O(\la^{\al_2-\de}).$$ 
Thus, $\PP_\om^t(F^\la_{0,\de}(0))\leq \PP^t_\om(F^\la_{0,\de}(1^+))=O(\la^{\al_2-\de})$. The proof for $\PP_\om^t(F^\la_{0,h}(0,0,1))$ is similar.\\
$(iii)$ 
If $\om\in R$, then, $\PP^t_{\om}(X^\la_{T^\la_{\de}}\in R)\geq \PP^t_{\om}(F^\la_{0,\de}(0))=1-O(\la^{\al_2-\de})$, where the last equality 
holds by $(ii)$. Suppose that $\om\notin R$. By $(ii)$,
$\limla \PP^t_\om (F^\la_{0,\de}(0))=1$, so that: 
$$ \limla \PP^t_{\om}(X^\la_{T^\la_\de}\in R)= \limla \PP^t_{\om}(X^\la_{T^\la_\de}\in R \, | \, F^\la_{0,\de}(0))=\mu^{[1]}(\om,R).$$ \hfill $\square$\\
The main consequences of Claim~\ref{easy} are that, combined with \eqref{disj}, it yields, as $h$ tends to $0$:
\begin{eqnarray}\label{imp}\limla \PP^t_\om(X^\la_{(t+h)/\la }\in R')&=&\limla \PP^t_\om\left(X^\la_{(t+h)/\la} \in R' \cap F^\la_{\de,h-\de}(0,1,0)\right)+o(h),\\
\label{imp2}\limla \PP^t_\om(X^\la_{(t+h)/\la }\in R)&=&1- \sum_{R'\neq R}\limla \PP^t_\om\left(X^\la_{(t+h)/\la} \in R'\right)+o(h). 
\end{eqnarray}
We will need the following coupling result which implies that, up to an error which vanishes with $\la$, 
the distribution at stage $T^\la_\de$ is $\pi^{[1],R}$.
\begin{claim}\label{tech}
If $R$ is aperiodic then, conditional to $\{X_{t/\la}\in \om\in R\}$ and $F^\la_{0,\de}(0)$, the distance in total variation between the distribution of $X^\la_{T^\la_\de}$ and 
$\pi^{[1],R}$ is $O(\la^\ep)$ 
as $\la$ tends to $0$.
\end{claim}
\noindent \textit{Proof of Claim~\ref{tech}.}
Let $P_\la^{[1]}$
and $\Ph^{[1]}$ be the restrictions of $P_\la^{[1]}$ and $\Ph^{[1]}$ to $R$ respectively. Let 
$S^{[1]}$ be a diagonal matrix such that $S^{[1]}(\om,\om):=\frac{1}{\la^{\al_1}}P_\la^{[1]}(\om,\om^c)$, for all 
$\om\in R$. It does not depend on $\la$ and that, by Gershgorin Circle Theorem,
all its eigenvalues have nonnegative real part. By construction $\Id-P_\la^{[1]}=\la^{\al_1} S^{[1]}(\Id - \Ph^{[1]})$. 
Thus, $\rho$ is an eigenvalue of $S^{[1]}(\Id - \Ph^{[1]})$ if and only if 
$1-\rho\la^{\al_1}$ is an eigenvalue of $P_\la^{[1]}$.
By aperiodicity, $1$ is a simple eigenvalue of $P_\la^{[1]}$, so that the second largest eigenvalue is
$1-\rho \la^{\al_1}$ for some eigenvalue of $S^{[1]}(\Id - \Ph^{[1]})$, $\rho\neq 0$.
By Perron-Frobenius Theorem, 
the distance in total variation between the two distributions is thus of order
$|1-\rho \la^{\al_1}|^{\la^{-\de}}\sim_{\la\to 0} \exp(-\eta\la^{\al_1-\de})$, which is $O(\la^\ep)$ for any $\ep>0$ by the choice of $\de$.

\hfill $\square$

\begin{claim}\label{rr'} 
For any $t\geq 0$, $h>0$ and $\om\in R$ we have, as $h$ tends to $0$: \begin{itemize}
 \item [$(i)$] $\limla \PP^t_{\om}\left(X^\la_{(t+h)/\la}\in R'\right)=A^{[1]}(R,R')h+o(h)$;
 \item [$(ii)$] $\limla \PP^t_{\om}\left(X^\la_{(t+h)/\la}\in R\right)=1+A^{[1]}(R,R)h+o(h)$.
\end{itemize}
\end{claim}
\noindent \textit{Proof of Claim~\ref{rr'}.} 
Assume first that $R$ is aperiodic. Thanks Claim \eqref{easy}-$(ii)$ and Claim~\ref{tech}, we can define some auxiliary random variable $\widetilde{X}^\la_\de$
distributed as $\pi^{[1],R}$ and such that 
$\PP(\widetilde{X}^\la_\de \neq X^\la_{T^\la_\de})=O(\la^{1-\de})$ as $\la$ tends to $0$.
Thus, up to an error which vanishes with $\la$, the distribution at stage $T^\la_\de$ is $\pi^{[1],R}$. Combining this coupling result with
\eqref{imp} yields, as $h$ tends to $0$:
\begin{eqnarray}\label{imp'}\limla \PP^t_\om(X^\la_{(t+h)/\la }\in R')&=&\limla \PP^t_{\pi^{[1],R}}\left(X^\la_{(t+h)/\la} \in R', \, F^\la_{0,h}(0,1,0)\right)+o(h).
\end{eqnarray}
To compute the right-hand-side of \eqref{imp'}, consider the following disjoint union:
$${F}^\la_{0,h}(0,1,0)=\bigcup_{\om'\notin R} \bigcup_{m=T^\la_\de}^{T^\la_{h-\de}}{F}^\la_{0,h}(m,\om'),$$
where ${F}^\la_{0,h}(m,\om')$ is the event of a secondary transition occurring at stage $m$, and not before nor after, to a state $\om'\notin R$. 
Notice that, by the choice of $\pi^{[1],R}$ and Claim~\ref{easy}, for any $m\in[T^\la_{0},T^\la_{h}]$:
$$\PP^t_\om({F}^\la_{0,h}(m,\om'))=(1-O(h))^2\sum_{\om\in R} \pi^{[1],R}Q_\la(\om,\om')=\sum_{\om\in R} \pi^{[1],R}Q_\la(\om,\om')+o(h).$$
On the other hand, by Claim \eqref{easy}-$(iii)$, for any $m\leq T^\la_{h-\de}$:
\begin{equation}\label{rff}\limla \PP^t_\om\left(X^\la_{(t+h)/\la}\in R' \, | \, {F}^\la_{0,h}(m,\om')\right)=\mu^{[1]}(\om',R').
\end{equation}
Consequently, as $h$ tends to $0$:
\begin{eqnarray*}
\limla \PP^\la_{\pi^{[1],R}}\left(X^\la_{T^\la_h} \in R', \, {F}^\la_{0,h}(0,1,0)\right)&=&
\limla \sum_{m=T^\la_\de}^{T^\la_{h-\de}}\sum_{\om'\notin R} \PP^\la_{\pi^{[1],R}}\left(X^\la_{T^\la_h}\in R', \, {F}^\la_{0,h}(m,\om')\right),\\
&=& \limla \left(\frac{h}{\la}-\frac{2}{\la^\de}\right)\left(\sum_{\om\in R, \, \om'\notin R}\pi^{[1],R}(\om)Q_\la(\om,\om')\mu^{[1]}(\om',R')+o(h)\right),\\
&=& \limla \left(h+O(\la^{1-\de})\right)\left( \frac{1}{\la}\sum_{\om'\in \Om^{[1]}\backslash R}Q^{[1]}_\la(R,\om')\mu^{[1]}(\om',R')+o(h)\right),\\
&=& hA^{[1]}(R,R')+o(h), 
\end{eqnarray*}
which gives $(i)$. Finally, $(ii)$ is now a consequence of the \eqref{imp2}.
The case where $R$ is periodic, needs minor changes. Note that periodicity can only happen if $\al_1=0$ for otherwise with positive probability the chain does not change of state.
Now, $R$ is then a disjoint union of $R^1\cup\dots \cup R^d$, and the restriction of $P_\la^{[1]}$ to these sets is aperiodic, with invariant measure $\pi^{[1],R}_k$ $(k=1,\dots,d$). One needs to take into account the subclass at time $T^\la_\de$ and define, in  the aperiodic case, some auxiliary random variable $\widetilde{X}^\la_{\de,k}$
distributed as $\pi^{[1],R}_k$ and such that $\PP(\widetilde{X}^\la_{\de,k} \neq X^\la_{T^\la_\de})=O(\la^{1-\de})$ as $\la$ tends to $0$. The results  then follows from the fact that $\pi^{[1],R}=\frac{1}{d}\sum_{k=1}^{d}\pi^{[1],R}_k$ and that, under the initial probability $\pi^{[1],R}_k$, the distribution at stage $m$ is $\pi^{[1],R}_{k+m}$, which is a shortcut for $\pi^{[1],R}_{k+m\,(\text{mod}\, d)}$.
The computation is now, for some $k$:
\begin{eqnarray*}
\limla \PP_{\pi^{[1],R}_k}\left(X^\la_{T^\la_h} \in R', \, {F}^\la_{0,h}(0,1,0)\right)&=&
\limla \sum_{m=T^\la_\de}^{T^\la_{h-\de}}\sum_{\om'\notin R} \PP_{\pi^{[1],R}_{k+m}}\left(X^\la_{T^\la_h}\in R', \, {F}^\la_{0,h}(m,\om')\right),\\
&=& \limla 
\frac{h}{\la }\frac{1}{d}\sum_{k=1}^d\left(\sum_{\om\in R, \, \om'\notin R} \pi^{[1],R}_{k}(\om)Q_\la(\om,\om')\mu^{[1]}(\om',R')+o(h)\right),\\
&=& h \limla \left( \frac{1}{\la} \sum_{\om\in R, \, \om'\notin R} \pi^{[1],R}(\om)Q_\la(\om,\om')\mu^{[1]}(\om',R')+o(h)\right),\\
&=& hA^{[1]}(R,R') +o(h),
\end{eqnarray*}
which proves the Claim. \hfill $\square$\\
Let us go back to the proof of Proposition~\ref{lp2}. Let $R=R^1\cup\dots R^d$ be a recurrence class of period 
$d\geq 1$. Consider four deterministic times as in Figure $2$, with $h>0$ and $t=0$, i.e. 
$$T^\la_0:=1, \quad T^\la_{\de}:= \frac{1}{\la^\de}, \quad T^\la_{h-\de}:=\frac{h}{\la}-
\frac{1}{\la^\de}, \quad T^\la_h:= \frac{h}{\la}.$$
For any $m\in \NN$, let $T^\la_h+m:=\lfloor h/\la \rfloor + m$. 
On the one hand, from Claim~\ref{rr'} (see Remark~\ref{flex}) one deduces that for any $R'\in \R^{[1]}$ and 
$\de'$ satisfying $1>\de'>\al_1$:
$$\limla \PP^\la_\om\left(X_{T^\la_{h-{\de'}}}\in R\, | \, X^\la_{T^\la_\de}\in R'\right)=\e^{A^{[1]}h}(R',R),$$
which, together with Claim~\ref{easy}-$(iii)$, yields:
\begin{equation}\label{dd2}\limla \PP^\la_\om\left(X_{T^\la_{h-\de'}}\in R\right)=\sum_{R'\in \R^{[1]}} \mu^{[1]}(\om,R')
\e^{A^{[1]}h}(R',R).
\end{equation}
By periodicity, if  $X^\la_{T^\la_{h-{\de'}}}\in R^r\subset R$, then $X^\la_{T^\la_h}\in R^{r+\lfloor 1/\la^{\de'}\rfloor \, (\mathrm{mod} \, d)}$.
Consequently, for any $\la>0$, $r=1,\dots,d$ and $D\in \NN^*$:
\begin{equation}\label{d1}\sum_{m=1}^{Dd} \ind_{\{ X^\la_{T^\la_h+m} \in R^{[k]}\}}=D.
\end{equation}
Thus, by Perron-Frobenius Theorem, for any $\om' \in R^{[k]}\subset R$, and $k=1,\dots, d$:
\begin{equation}\label{dd3}\limla \frac{1}{Dd}\sum_{r=1}^{Dd} \PP^\la_\om\left(X^\la_{T^\la_h+r} 
= \om' \, | \, X^\la_{T^\la_{h-\de'}}\in R\right)=\frac{D}{Dd}\pi_r^{[1],R}(\om')=\pi^{[1],R}(\om'),
\end{equation}
using the fact that $\pi^{[1],R}_{r'}(\om')=0$, for all $r'\neq r$.
Combining \eqref{dd2} and \eqref{dd3} one has, thanks to the definition of $N$, that for any $\om \in R\in \R^{[1]}$:
\begin{eqnarray}\label{d3}\limla \frac{1}{N}\sum_{r=1}^N \PP^\la_\om\left(X^\la_{T^\la_h+r} 
= \om'\right)&=&\sum_{R'\in \R^{[1]}} \mu^{[1]}(\om,R')\e^{A^{[1]}h}(R',R) \pi^{[1],R}(\om'),\\
&=& \mu^{[1]}\e^{A^{[1]}h}M^{[1]}(\om,\om'),
\end{eqnarray}
which proves Proposition~\ref{lp2}. \hfill $\square$ 

\subsection{Proof of Theorem~\ref{main} and Algorithm}\label{proof}
Theorem~\ref{main} can be proved using the same ideas, inductively. 
The first step is precisely Proposition~\ref{lp2}. If $Q^{[1]}_\la$ is not critical, proceed by steps. 
The aggregation of states is illustrated in Figure $3$.\\ 
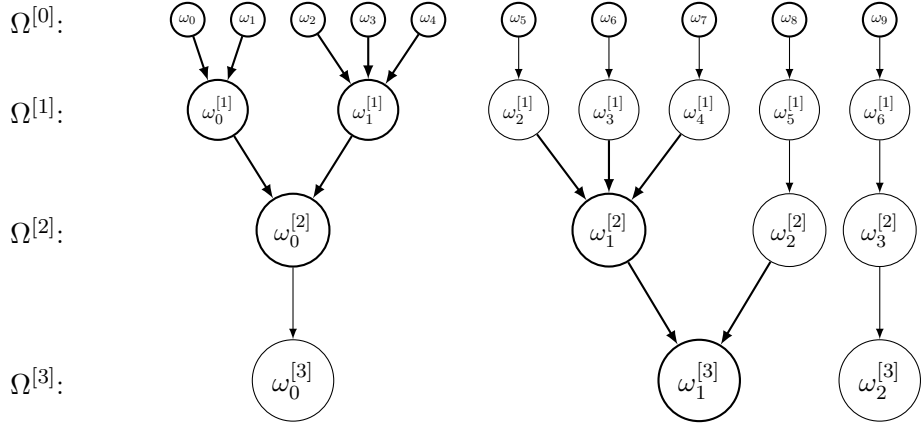
\begin{figure}\label{fig3}
\begin{center}
\begin{tikzpicture}[xscale=0.8, yscale=0.8]
\node  at (-7.5,0) {$\Om^{[0]}$:};
\node  at (-7.5,-1.5) {$\Om^{[1]}$:};
\node at (-7.5,-3.5) {$\Om^{[2]}$:};
\node at (-7.5,-6) {$\Om^{[3]}$:};
\node[draw,thick,circle,scale=0.55](om0) at (-5,0) {$\om_0$};
\node[draw,thick,circle,scale=0.55](om1) at (-4,0) {$\om_1$};
\node[draw,thick,circle,scale=0.55](om2) at (-3,0) {$\om_2$};
\node[draw,thick,circle,scale=0.55](om3) at (-2,0) {$\om_3$};
\node[draw,thick,circle,scale=0.55](om4) at (-1,0) {$\om_4$};
\node[draw,thick,circle,scale=0.55](om5) at (0.5,0) {$\om_5$};
\node[draw,thick,circle,scale=0.55](om6) at (2,0) {$\om_6$};
\node[draw,thick,circle,scale=0.55](om7) at (3.5,0) {$\om_7$};
\node[draw,thick,circle,scale=0.55](om8) at (5,0) {$\om_8$};
\node[draw,thick,circle,scale=0.55](om9) at (6.5,0) {$\om_9$};
\node[draw,thick,circle,scale=0.7](R0) at (-4.5,-1.5) {$\om^{[1]}_0$};
\node[draw,thick,circle,scale=0.7](R1) at (-2,-1.5)  {$\om^{[1]}_1$};
\node[draw,circle,scale=0.7](R2) at (0.5,-1.5) {$\om^{[1]}_2$};
\node[draw,circle,scale=0.7](R3) at (2,-1.5) {$\om^{[1]}_3$};
\node[draw,circle,scale=0.7](R4) at (3.5,-1.5) {$\om^{[1]}_4$};
\node[draw,circle,scale=0.7](R5) at (5,-1.5) {$\om^{[1]}_5$};
\node[draw,circle,scale=0.7](R6) at (6.5,-1.5) {$\om^{[1]}_6$};
\node[ draw,thick,circle, scale=.85](S0) at (-3.25,-3.5) {$\om^{[2]}_0$};
\node[draw,thick,circle, scale=.85](S1) at (2,-3.5) {$\om^{[2]}_1$};
\node[draw,circle, scale=.85](S2) at (5,-3.5) {$\om^{[2]}_2$};
\node[draw,circle, scale=.85](S3) at (6.5,-3.5) {$\om^{[2]}_3$};
\draw[->, thick,>=latex] (om0) to (R0); \draw[->, thick,>=latex] (om1) to (R0);
\draw[->, thick,>=latex] (om2) to (R1); \draw[->, thick,>=latex] (om3) to (R1);
 \draw[->, thick,>=latex] (om4) to (R1);
 \draw[->,thin,>=latex] (om5) to (R2);  \draw[->,thin,>=latex] (om6) to (R3);
 \draw[->,thin,>=latex] (om7) to (R4);  \draw[->,thin,>=latex] (om8) to (R5);
 \draw[->,thin,>=latex] (om9) to (R6);
\node[draw,circle,scale=0.95](U0) at (-3.25, -6) {$\om^{[3]}_0$};
\node[draw,circle,thick,scale=0.95](U1) at (3.5, -6) {$\om^{[3]}_1$};
\node[draw,circle,scale=0.95](U2) at (6.5, -6) {$\om^{[3]}_2$};
\draw[->,thick,>=latex] (R0) to (S0);
\draw[->,thick,>=latex] (R1) to (S0);
\draw[->,>=latex] (S0) to (U0);
\draw[->,thick,>=latex] (R2) to (S1);
\draw[->,thick,>=latex] (R3) to (S1);
\draw[->,>=latex,thick] (S1) to (U1);
\draw[->,thick,>=latex] (R4) to (S1);
\draw[->,>=latex](R5) to (S2);
\draw[->,>=latex](R6) to (S3);
\draw[->,>=latex](S3) to (U2);
\draw[->,>=latex,thick](S2) to (U1);
\end{tikzpicture}
\end{center}
\caption{Example of aggregation of states, for $k=3$. 
Since we only aggregate recurrence classes, we may deduce from the diagram that
 $\om^{[1]}_0, \om^{[1]}_1\in \R^{[1]}$,  $\om^{[2]}_0, \om^{[2]}_1\in \R^{[2]}$ and 
$\om^{[3]}_1 \in \R^{[3]}$. Recurrent states are indicated with a thicker border ($\R^{[0]}=\Om^{[0]}=\Om$ by definition). 
The diagram does not tell whether any of the other states are recurrent or transient, in their  corresponding  state spaces.} 
\end{figure}

\noindent\textbf{Initialisation (Step $0$).} Let $\Om^{[-1]}:=\emptyset$, $\R^{[0]}=\Om^{[0]}=\Om$ and $\T^{[0]}:=\emptyset$. 
Let $Q^{[0]}_\la:=Q_\la$,
$\pi^{[0],\om}=\de_{\om}$, for any $\om\in \Om$. The coefficients and exponents $c^{[0]}_{\om,\om'}, e^{[0]}_{\om,\om'},
c^{[0],\om}_{\om'}, e^{[0],\om}_{\om'}$ ($\om,\om'\in \Om$) are deduced from the definitions of $Q^{[0]}_\la$ and $\pi^{[0],\om}$.\\

\noindent \textbf{Induction (Step $k$, $k\geq1$).} The following quantities have already been defined, or computed, for $\ell=0,\dots,k-1$:
 $\R^{[\ell]}$, $\T^{[\ell]}$, $\Om^{[\ell]}$, 
$P^{[\ell]}_\la$, $Q^{[\ell]}_\la$, $\mu^{[\ell]}(u,R)$, 
$c^{[\ell]}_{u,v}$, $e^{[\ell]}_{u,v}$, $\pi^{[\ell],R}(w)$, $c^{[\ell],R}_w$, $e^{[\ell],R}_w$, 
for any $u,v\in \Om^{[\ell]}$, $w\in \Om^{[ \ell-1]}$ and $R\in \R^{[\ell]}$. 
Define $\al_{k}$ as follows:
\begin{equation}\label{defalk}\al_{k}:=\min\{e^{[k-2]}_{u,v}+e^{[k-1],R}_u \, | \, u\in \Om^{[k-2]}, \, u\in R\in \R^{[k-1]}, v\in \Om^{[k-2]}\backslash R \}.
\end{equation}
Note that this definition coincides with $\al_1$ and $\al_2$ (defined in Section~\ref{fast}) for $k=1,2$.
Define a stochastic matrix by setting $P^{[k]}_\la(\om,\om'):=Q_\la(\om,\om')\ind_{\{e_{\om,\om'}\leq \al_{k}\}}$ for all $\om'\neq \om\in \Om$. 
Compute its recurrence classes $\R^{[k]}$, its invariant measures $\pi_\la^{[k],R}$ and its transients states $\T^{[k]}$, and 
define $\Om^{[k]}:=\R^{[k]}\cup \T^{[k]}$. As in Corollary~\ref{jaff}, there exists $c^{[k],R}_\om>0$ and 
$e^{[k],R}_\om\in \{\al_{k}-\al_i \, | \, i=0,\dots,k\}$, for all $\om\in \Om^{[k-1]}$, $\om\in R\in \R^{[k]}$, such that:
$$\pi_\la^{[k],R}(\om)\sim_{\la\to 0}c^{[k],R}_\om \la^{e^{[k],R}_\om}.$$
Define the aggregated stochastic matrix $Q^{[k]}_\la$ over $\Om^{[k]}$ by setting:
\begin{equation}\label{defQk}
Q^{[k]}_\la(\om,\om') :=\begin{cases} Q_\la(\om,\om') & \text{ if } \om,\om' \in \T^{[k]};\\
                 \sum_{z'\in \om'} Q_\la(\om,z') & \text{ if } \om \in \T^{[k]}, \text{ and } \om'\in \R^{[k]};\\
                 \sum_{z\in \om} \pi^{[k],\om}(z)Q_\la(z,\om') & \text{ if } \om \in \R^{[k]}, \text{ and } \om'\in \T^{[k]};\\
                 \sum_{z\in \om, \, z'\in \om'}\pi^{[k],\om}(z)Q_\la(z,z') & \text{ if } \om,\om' \in \R^{[k]};
                \end{cases}
\end{equation}
Deduce $c^{[k]}_{\om,\om'}$ and $e^{[k]}_{\om,\om'}$ from \eqref{defQk}, for all $\om,\om'\in \Om^{[k]}$. 
If, for instance, $\om,\om'\in \R^{[k]}$:
\begin{eqnarray}e^{[k]}_{\om,\om'}&=&\min_{z\in \om, z'\in \om'} e^{[k-1]}_{z,z'}+ e^{[k],\om}_z ,\\
 c^{[k]}_{\om,\om'}&=&
\sum_{z\in \om, z'\in \om'}c^{[k-1]}_{z,z'}c^{[k],\om}_{z}\ind_{\{ e^{[k-1]}_{z,z'}+e^{[k],\om}_z=e^{[k]}_{\om,\om'} \}}.
\end{eqnarray}
\noindent \textbf{If $\al_{k}<1$, let $k:= k+1$ and go back to Step $k$.} \\ 
\noindent \textbf{If $\al_{k}\geq 1$, terminate.}\\
Now, the matrix $Q_\la^{[k]}$ is critical, so that the result of Section~\ref{crit} apply.
By construction, $\Om^{[\ell]}$ is a partition of $\Om^{[\ell-1]}$ for any $\ell=1,\dots,k$ (see Figure $3$). 
Thus, for any $\om\in \Om$, there exists a unique sequence $\om=\om^{[0]},\om^{[1]}, \dots,\om^{[k]}$ such that:
\begin{equation}\label{seq}\om^{[\ell]}\in \Om^{[\ell]}, \quad \text{and} \quad \om^{[\ell-1]}\in \om^{[\ell]}, \ \text{ for all } \ell=0,\dots, k.
\end{equation}
In particular, there exists a unique $\om^{[k]}\in \Om^{[k]}$ which contains the initial state 
$\om$. Define the entrance distribution $\mu:\Om\to \De(\R^{[k]})$ (see \eqref{entr}) setting,
for each $R\in \R^{[k]}$ and $\om\in \Om$:
$$\mu(\om,R)=\limn (P^{[k]}_\la)^n(\om,R).$$ 
Define the infinitesimal generator
$A:\R^{[k]}\times \R^{[k]}\to [0,\infty)$ (see \eqref{defA1}) by setting, for any $R\neq R'\in \R^{[k]}$:
\begin{eqnarray}\label{defAk}A^{[k]}(R,R')
&:=& \limla \frac{1}{\la}\left(\sum_{\om \in \Om^{[k]}\backslash R} Q^{[k]}_\la(R,\om)\mu(\om,R')\right),
\end{eqnarray}
and $A^{[k]}(R,R)=-\sum_{R'\neq R} A^{[k]}(R,R')$.
Finally, let $M:\R^{[k]}\to \De(\Om)$ be such that, for each $\om\in \Om$ and $R\in \R^{[k]}$ 
(where $\om=\om^{[0]},\om^{[1]}, \dots,\om^{[k]}=R$, satisfying \eqref{seq}):
$$M(R,\om):=\limla \prod_{\ell=1}^{[k]} \pi_\la^{[\ell],\om^{[\ell]}}(\om^{[\ell-1]}).$$
We thus obtain $\mu$, $A$ and $M$ from which $P_t$ can be computed, for all $t\geq 0$. We finish this section 
by justifying $P_t=\mu \e^{At}M$.
\begin{claim}\label{last}
$P_t$ exists for all $t\geq 0$. For any $\om,\om'\in \Om$:
\begin{equation}\label{final}
P_t(\om,\om'):=\limla \frac{1}{N}\sum_{r=1}^{N} \PP^\la_\om\left(X^\la_{t/\la+r} 
= \om' \right)=\mu^{[k]}\e^{A^{[k]}t}M^{[k]}(\om,\om').
\end{equation}
 \end{claim}
\begin{proof}
For any $0<\de<1$, let $T(\la,\de):=t/\la-1/\la^\de$. 
Let $\om'=\om^{[0]},\om^{[1]}, \dots,\om^{[k]}$ be a sequence satisfying \eqref{seq}. 
Let $\de_\ell$ ($\ell=1,\dots,k-1$) satisfy: $$0\leq \al_1<\de_1<\al_2<\dots<\al_{k-1}<\de_{k-1}<1\leq \al_k.$$ 
In particular, $t/\la\gg T(\la,\de_1) \gg\dots \gg T(\la,\de_{k-1})$. 
On the one hand, by Proposition~\ref{lp} (see Remark~\ref{flex}), for any $t>0$ and $R\in \R^{[k]}$: 
\begin{eqnarray}\label{insat}
\limla 
\PP^\la_\om\left(X^\la_{T(\la,\de_{k-1})} \in R\right) 
&=& 
  P^{[k]}\e^{A^{[k]}t}(\om,R).
  \end{eqnarray}
$$\limla \PP^\la_\om\left(X^\la_{T(\la,\de_{\ell-1})} 
= \om^{[\ell-1]} \, | \, X^\la_{T(\la,\de_{\ell})}\in \om^{[\ell]} \right)=
\limla  \pi_\la^{[\ell],\om^{[\ell]}}(\om^{[\ell-1]}).$$
Thus, multiplying \eqref{insat} and the latter over all $\ell=2,\dots,k$ 
yields:
\begin{eqnarray}\label{sat}
\limla
\PP^\la_\om\left(X^\la_{T(\la,\de_{1})} = \om^{[1]}\right)
&=& 
  \mu^{[k]}\e^{A^{[k]}t}(\om^{[k]})\left(\limla \prod_{\ell=2}^{[k]} \pi_\la^{[\ell],\om^{[\ell]}}(\om^{[\ell-1]})\right).
  \end{eqnarray}
To avoid the periodicity issues of the first 
order transitions, it is enough to consider the times $t/\la+1, \dots,t/\la+N$. One obtains, exactly in the same way as
 in \eqref{d1} and \eqref{dd3}:
\begin{equation}\limla \frac{1}{N}\sum_{r=1}^{N} \PP^\la_\om\left(X^\la_{t/\la+r} 
= \om^{[0]} \, | \, X^\la_{T(\la,\de_{1})}\in \om^{[1]}\right)=\pi^{[1],\om^{[1]}}(\om^{[0]}),
\end{equation}
for each $\om^{[0]}\in \Om$ and $\om^{[1]}\in \R^{[1]}$, which concludes the proof.
\end{proof}
\subsection{Relaxing Assumption 1}
Though quite natural, Assumption $1$ can be relaxed by noticing that $\mu$, $A$ and $M$, and consequently, $P_t$, 
depend only on the relative speed of convergence of the mappings $\la\mapsto Q_\la(\om,\om')$, for $\om\neq \om'$ and 
$\la\mapsto \la$, and of some products between them. 
Define: 
$$\F_Q:=\{\la\mapsto \la, (\la\mapsto Q_\la(\om,\om')), \ \om\neq \om\in \Om \}.$$
One can then replace Assumption $1$ by:\\ 
\textbf{Assumption 1':} For any $A,B\subset \F_Q$, the $\limla \frac{\prod_{a\in A} a(\la)}
{\prod_{b \in B} b(\la)}$ exists in $[0,\infty]$.\\
To perform the algorithm described in Section~\ref{proof} it is enough to use \cite[Proposition 2]{OB14}, which implies that 
if $(Q_\la)_\la$ satisfies Assumption 1', there exists coefficients and exponents
$(c_a,e_a)_{a\in \F_Q}$ such that:
\begin{equation}\label{ass'}\limla \frac{\displaystyle\prod_{a \in A} a(\la)}
{\displaystyle\prod_{b \in B} b(\la)}=\limla \frac{\displaystyle\prod_{a \in A} 
c_a \la^{e_a}}
{\displaystyle\prod_{b\in B} c_b \la^{e_b}}, \quad \forall A,B\subset \F_Q.
\end{equation}
The coefficient and exponent corresponding to $\la\mapsto \la$ are, of course, equal to $1$.

\subsection{Illustration of the algorithm}\label{ill}
\begin{figure}\label{fig4}
\begin{center}
\begin{tikzpicture}
\node [above] at (4,2.5) {Initial state};
\node[draw,circle,scale=0.7](1) at (0,0) {$1$};
\node[draw,circle,scale=0.7](2) at (1,0) {$2$};
\node[draw,circle,scale=0.7](3) at (2,0) {$3$};
\node[draw,circle,scale=0.7](4) at (4,0) {$4$};
\node[draw,circle,scale=0.7](5) at (4,2) {$5$};
\node[draw,circle,scale=0.7](6) at (4,1) {$6$};
\node[draw,circle,scale=0.7](7) at (7,0) {$7$};
\node[draw,circle,scale=0.7](8) at (8,0) {$8$};
\draw [->,>=latex] (1) [bend right=-50] to (2);
\draw [->,>=latex] (2) [bend right=-50] to (3);
\draw [->,>=latex] (3) [bend right=-50] to (1);
\draw [->,>=latex] (6) [bend right=25] to (4);
\draw [->,>=latex] (4) [bend right=50] to (5);
\draw [->,>=latex] (7) [bend right=-50] to (8);
\draw [->,>=latex] (8) [bend right=-50] to (7);
\draw [->,>=latex] (5) [bend right=15] to (0,0.25);
\draw [->,>=latex] (5) [bend right=25] to (6);
\draw [->,>=latex] (5) [bend right=-15] to (8,0.25);
\node at (2,1.8) {$\frac{1}{3}$}; \node at (3.6,1.5) {$\frac{1}{3}$};
\node at (6,1.8) {$\frac{1}{3}$};
\node [scale=0.9]at (7.5,-0.6) {$1$};
\node [scale=0.9]at (7.2,0.5) {$1$};
\node [scale=0.9] at (1,-0.8) {$e\la^{3/5}$};
\node [scale=0.8] at (0.9,0.6) {$a\la^{1/5}$};
\node [scale=0.8] at (1.85,0.6) {$b\la^{2/5}$};
\node [scale=0.9] at (4.9, 1) {$g\la$};
\node [scale=0.8] at (3.3,0.5) {$d\la^{1/5}$};
\draw [->,>=latex] (1) [bend right=80] to (4);
\draw [->,>=latex] (2) [bend right=-70] to (4);
\node [scale=0.8] at (2.7,1.3) {$f\la^{4/5}$};
\node [scale=0.8] at (3,-1) {$e\la^{3/5}$};
\end{tikzpicture}
\end{center}
\caption{Illustration of $Q_\la^{[0]}$.} 
\end{figure}
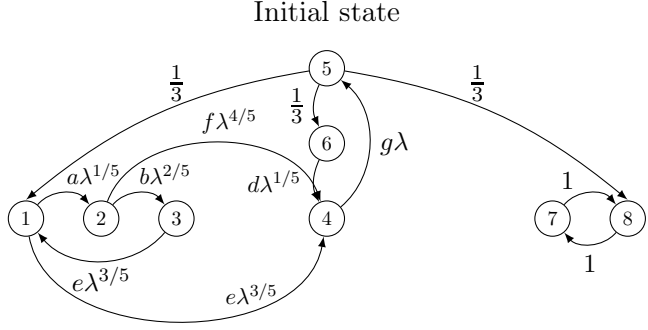

Suppose that $Q_\la$ is a stochastic matrix over $\Om=\{1,\dots,8\}$ satisfying Assumption $1$, such
that, for some $a,b,c,d,e,f,g,h>0$: 
%
\[Q_\la\sim_{\la\to0}\begin{pmatrix} 
1 & a\la^{1/5} & 0 &e\la^{3/5} & 0& 0  & 0& 0\\ 
0 & 1 & b\la^{2/5} & f\la^{4/5}& 0 & 0  & 0& 0\\
c\la^{3/5} & 0 & 1 &0 & 0& 0  & 0& 0\\
0 & 0& 0  &1 & g\la & 0  & 0& 0\\
1/3 & 0 & 0 & 0&  0 & 1/3  & 0 & 1/3\\
0 & 0 & 0 &d \la^{1/5} & 0& 1  & 0& 0\\
0 & 0 & 0 &0 & 0& 0  & 0 & 1\\
0 & 0 & 0 &0 & 0& 0  & 1 & 0
        \end{pmatrix}
\]
See Figure $4$ for an illustration. For simplicity, let us fix an initial, say $5$, and compute $P_t(5,k)=\limla Q_\la^{t/\la}(5,k)$, for any $t>0$ and $k\in \Om$.
We use the definition \eqref{defalk} to compute $0\leq \al_1\leq \dots \leq \al_k$. \\
\textbf{Step $1$}. $\al_1=0$. 
$$\R^{[1]}=\{1,2,3,4,6,u\} \text{ and } \T^{[1]}=\{5\},$$
where $u:=\{7,8\}$ is a $2$-periodic recurrence class. Computes the (nontrivial) entrance law $P^{[1]}(5,1)=
P^{[1]}(5,6)=P^{[1]}(5,u)=1/3$ 
and the invariant measure $\pi^{[1],u}=\frac{1}{2}\de_7+\frac{1}{2}\de_8$. 
Since there are non-trivial recurrence classes, one defines the aggregated matrix $Q^{[1]}_\la$.\\
\textbf{Step $2$}. $\al_2=1/5$. 
The transitions of order $\al_2$ are $1\mapsto 2$ and $6\mapsto 4$.
$$\R^{[2]}=\{2,3,4,u\} \text{ and } \T^{[2]}= \{1,5,6\}.$$ 
\noindent \textbf{Step $3$.} $\al_3=2/5$. 
The only transition of order $\al_3$ is $2\mapsto 3$. 
$$\R^{[2]}=\{3,4,u\} \text{ and } \T^{[2]}=\{1,2,5,6\}.$$ 
\noindent \textbf{Step $4$.} $\al_4=3/5$. 
The only transition of order $\al_4$  is $3\mapsto 1$. The subset $v:=\{1,2,3\}$ is now a recurrence class
$$\R^{[4]}=\{v,4,u\} \text{ and } \T^{[4]}=\{5,6\}.$$ 
 Compute the invariant measure $\pi^{[4],v}_\la$. 
Clearly, 
$\widehat{\pi^{[4],v}}=\frac{1}{3}\de_1+\frac{1}{3}\de_2+\frac{1}{3}\de_3$, for it is a cycle. By Corollary~\ref{jaff}:
\[\pi^{[4],v}_\la(1)\sim \frac{\frac{1/3}{a\la^{1/5}}}{\frac{1/3}{a\la^{1/5}}+\frac{1/3}{b\la^{2/5}}+\frac{1/3}{c\la^{3/5}}}\sim \frac{c}{a}\la^{2/5},\]
and similarly state $2$ and $3$, so that $\pi_\la^{[4],v} \sim_{\la\to 0} \frac{c}{a}\la^{\frac{2}{5}}\de_1+\frac{c}{b}\la^{\frac{1}{5}}\de_2+\de_3$.
Since there are non-trivial recurrence classes, one defines the aggregated matrix $Q^{[4]}_\la$ (see Figure $5$).\\
\begin{figure}\label{fig5}
\begin{center}
\begin{tikzpicture}
\node [above] at (4,2.5) {Initial state};
 \node[draw,thick,circle,scale=1.2](v) at (1,0) {$v$};
\node[draw,circle,scale=1.2](4) at (4,0) {$4$};
\node[draw,circle,scale=0.8](5) at (4,2) {$5$};
\node[draw,circle,scale=0.8](6) at (4,1) {$6$};
\node[draw,thick, circle,scale=1.2](u) at (7.5,0) {$u$};
\draw [->,>=latex] (6) [bend right=25] to (4);
\draw [->,>=latex] (4) [bend right=50] to (5);
\draw [->,>=latex] (5) [bend right=15] to (v);
\draw [->,>=latex] (5) [bend right=25] to (6);
\draw [->,>=latex] (5) [bend right=-15] to (u);
\node at (2,1.6) {$\frac{1}{3}$}; \node at (3.6,1.5) {$\frac{1}{3}$};
\node at (6.1,1.6) {$\frac{1}{3}$};
\node  at (5, 1) {$g\la$};
\node [scale=0.8] at (3.3,0.6) {$d\la^{1/5}$};
\draw [->,>=latex] (v) [bend right=50] to (4);
\node  at (2.5,-1.2) {$\left(\frac{c}{a}e+\frac{c}{b}f\right)\la$};
\end{tikzpicture}
\end{center}
\caption{Illustration of $Q_\la^{[4]}$.}
\end{figure}
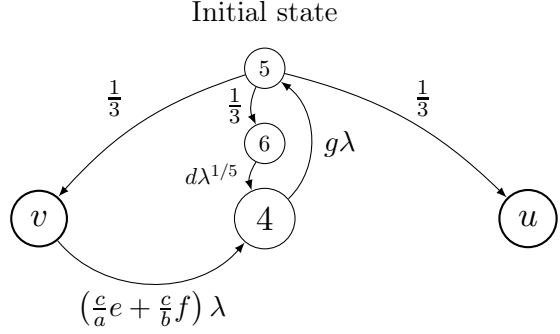
\noindent \textbf{Step $5$.} $\al_5=1$. Terminate.
Compute the infinitesimal generator over $\R^{[4]}$:
\[A=\begin{pmatrix}-\left(\frac{c}{a}e+\frac{c}{b}f\right) &  \phantom{-}\frac{c}{a}e+\frac{c}{b}f  & 
0\\
\frac{1}{3}g & -\frac{2}{3}g & \frac{1}{3}g\\ \phantom{-}0&\phantom{-}0&0 
\end{pmatrix}.\]
On the other hand, the entrance measure is $P=\frac{1}{3}\de_v+\frac{1}{3}\de_4+\frac{1}{3}\de_u$. 
Finally: $$M(v,\cdot)=\limla \pi_\la^{[4],v}=\de_3, \ M(4,\cdot)=\de_4 \text{ and } M(u,\cdot)=\limla \pi^{[1],u}=\frac{1}{2}\de_7+\frac{1}{2}\de_8.$$
The periodicity yields $N=2$, so that for any $k\in \Om$:
\[P_t(5,k)=\limla \frac{1}{2}\sum_{r=1}^2 \PP^\la_5(X^\la_{t/\la+r}=k)=\mu \e^{At}M(5,k).\]

\bibliographystyle{amsplain} 
\bibliography{bibliothese2} 

\end{document}